\documentclass[10pt]{amsart}
\usepackage[a4paper, margin=1in, marginparwidth=0.7cm]{geometry}
\usepackage{amssymb}
\usepackage[centertags]{amsmath}
\usepackage{amsthm}
\usepackage{amsfonts}
\usepackage{eucal}
\usepackage{mathrsfs}

\usepackage[all,cmtip]{xy}
\usepackage{graphicx, eepic}

\usepackage{fancyhdr}

\usepackage[shortlabels]{enumitem} %
\usepackage{lmodern}
\usepackage[utf8]{inputenc}
\usepackage[T1]{fontenc}
\usepackage{xcolor}

\usepackage{bbm}
\usepackage{pdfpages}
\usepackage{titletoc}
\usepackage{booktabs}
\usepackage{mathtools}
\usepackage{thmtools}

\usepackage{epstopdf}
\usepackage{epsfig}
\usepackage{microtype}

\definecolor{mymy}{RGB}{0, 153, 153}
\usepackage[colorlinks=true,citecolor=cyan,linkcolor=magenta, urlcolor=violet, filecolor=cyan, backref=page]{hyperref} 

\usepackage{multirow, url}
\usepackage{textcomp}

\DeclareMathAlphabet{\cmcal}{OMS}{cmsy}{m}{n}
\newtheorem{thm}{Theorem}[section]
\newtheorem{cor}[thm]{Corollary}

\newtheorem{prop}[thm]{Proposition}

\newtheorem{conj}[thm]{Conjecture}

\theoremstyle{definition}
\newtheorem{defn}[thm]{Definition}

\theoremstyle{remark}
\newtheorem{rem}[thm]{\bf{Remark}}

\newtheorem{ques}[thm]{\bf{Question}}

\newtheorem{notation}[thm]{\bf{Notation}}

\numberwithin{equation}{section} \numberwithin{table}{section}
	
\newtheorem*{thm*}{\bf{Theorem}}
\newtheorem*{claim*}{\bf{Claim}}
\newtheorem*{rem*}{\bf{Remark}}
\newtheorem*{rems*}{\bf{Remarks}}
\newtheorem*{exam*}{\bf{Example}}
\newtheorem*{exams*}{\bf{Examples}}

\newcommand{\C}{{\mathbb{C}}}

\newcommand{\F}{{\mathbb{F}}}

\newcommand{\Q}{{\mathbb{Q}}}
\newcommand{\R}{{\mathbb{R}}}

\newcommand{\T}{{\mathbb{T}}}

\newcommand{\Z}{{\mathbb{Z}}}

\newcommand{\fa}{{\mathfrak{a}}}
\newcommand{\fb}{{\mathfrak{b}}}

\newcommand{\m}{{\mathfrak{m}}}
\newcommand{\fn}{{\mathfrak{n}}}

\newcommand{\fH}{{\mathfrak{H}}}

\newcommand{\cI}{{\cmcal{I}}}

\newcommand{\cO}{{\cmcal{O}}}

\def\d{\delta}

\def\g{\gamma}

\def\<{\langle}
\def\>{\rangle}

\newcommand{\zell}{{\Z/{\ell\Z}}}
\newcommand{\muell}{{\mu_{\ell}}}
\newcommand{\Gm}{{\mathbb{G}}_{m}}

\newcommand{\inj}{\hookrightarrow}

\newcommand{\arinj}{\ar@{^(->}}
\newcommand{\arsurj}{\ar@{->>}}
\newcommand{\arsub}{\ar@{}[r]|-*[@]{\subset}}
\newcommand{\arsup}{\ar@{}[r]|-*[@]{\supset}}
\newcommand{\arcap}{\ar@{}[d]|-*[@]{\subset}}
\newcommand{\arcup}{\ar@{}[u]|-*[@]{\subset}}
\newcommand{\arin}{\ar@{}[u]|-*[@]{\in}}

\renewcommand{\pmod}[1]{{\,(\operatorname{mod}\hspace{0.7mm} {#1})}}

\newcommand{\Hom}{{\operatorname{Hom}}}
\newcommand{\Gal}{{\operatorname{Gal}}}

\newcommand{\End}{{\operatorname{End}}}
\newcommand{\Aut}{{\operatorname{Aut}}}

\newcommand{\SL}{{\operatorname{SL}}}

\newcommand{\sheafhom}{{\mathscr{H}\kern-.5pt om}}
\newcommand{\sheafext}{{\mathscr{E}\kern-.5pt xt}}

\newcommand{\Frob}{{\operatorname{Frob}}}

\newcommand{\old}{{\operatorname{old}}}
\newcommand{\new}{{\operatorname{new}}}

\newcommand{\upto}{{up to products of powers of 2 and 3}}
\newcommand{\modl}{{\pmod {\ell}}}

\renewcommand{\th}{{\text{th}}}
\renewcommand{~}{\hspace*{0.3mm}}

\newcommand{\ts}{{\textsection}}
\newcommand{\ra}{\rightarrow}
\newcommand{\wt}{\widetilde}

\newcommand{\ov}{\overline}

\mathchardef\hyp="2D

\newcommand{\vv}{\vspace*{2mm}}

\makeatletter
\def\@tocline#1#2#3#4#5#6#7{\relax
  \ifnum #1>\c@tocdepth 
  \else
    \par \addpenalty\@secpenalty\addvspace{0.1em}%
    \begingroup \hyphenpenalty\@M
    \@ifempty{#4}{%
      \@tempdima\csname r@tocindent\number#1\endcsname\relax
    }{%
      \@tempdima#4\relax
    }%
    \parindent\z@ \leftskip#3\relax \advance\leftskip\@tempdima\relax
    \rightskip\@pnumwidth plus4em \parfillskip-\@pnumwidth
    #5\leavevmode\hskip-\@tempdima
      \ifcase #1
       \or\or \hskip 2em \or \hskip 4em \else \hskip 6em \fi%
      #6\nobreak\relax
    \dotfill\hbox to\@pnumwidth{\@tocpagenum{#7}}\par
    \nobreak
    \endgroup
  \fi}
\makeatother

\begin{document}                                                                          
\title{Non-optimal levels of a reducible mod $\ell$ modular representation}
\author{Hwajong Yoo}
\address{Center for Geometry and Physics, Institute for Basic Science (IBS), Pohang, Republic of Korea 37673}
\email{hwajong@gmail.com}
\thanks{This work was supported by IBS-R003-D1.}


\begin{abstract}
Let $\ell \geq 5$ be a prime and let $N$ be a square-free integer prime to $\ell$. For each prime $p$ dividing $N$, let $a_p$ be either $1$ or $-1$. We give sufficient criteria for the existence of a newform $f$ of weight 2 for $\Gamma_0(N)$ such that the mod $\ell$ Galois representation attached to $f$ is reducible and $U_p f = a_p f$ for primes $p$ dividing $N$. The main techniques used are level raising methods based on an exact sequence due to Ribet.
 
\end{abstract}
\maketitle
\setcounter{tocdepth}{1}
\tableofcontents

\section{Introduction}
It has been known that newforms for congruence subgroups of $\SL_2(\Z)$ give rise to compatible systems of $\ell$-adic Galois representations, 
and if the $\ell$-adic Galois representations attached to two newforms are isomorphic for some prime $\ell$, 
then the newforms are, in fact, equal. But the corresponding statement is not true for the semisimplifications of the mod $\ell$ reductions of $\ell$-adic Galois representations attached to newforms, as different newforms can be congruent modulo $\ell$. To study the different levels from which a given modular mod $\ell$ representation $\rho$ can arise is interesting and has been discussed by several mathematicians, Carayol, Diamond, Khare, Mazur, Ribet, and Taylor in the case where $\rho$ is (absolutely) irreducible. (For more details, see \cite{DT94}.) 

For simplicity, fix a prime $\ell\geq 5$ and let $f$ be a newform of weight 2 for $\Gamma_0(N)$ with a square-free integer $N$ prime to $\ell$. Assume that $\ov{\rho}_f$, the semisimplified mod $\ell$ 
Galois representation attached to $f$, is reducible. Then, $\ov{\rho}_f \simeq \mathbbm{1} \oplus \chi_{\ell}$, where $\mathbbm{1}$ is the trivial character and $\chi_{\ell}$ is the mod 
$\ell$ cyclotomic character (Proposition \ref{prop:reducible}). The Serre conductor (see \cite{Se87}) of $\mathbbm{1} \oplus \chi_{\ell}$ is 1 because it is unramified outside $\ell$. The main purpose of this article is to find possible non-optimal levels of $\mathbbm{1} \oplus \chi_{\ell}$ as in the irreducible case due to Diamond and Taylor \cite{DT94}.
(For other works of non-optimal levels of reducible mod $\ell$ Galois representations, see \cite{BM1, BM2}.)
Since we consider a newform $f$ of weight 2 and square-free level $N$ with trivial character, an eigenvalue of the Hecke operator $U_p$ of $f$ is either $1$ or $-1$ for a prime divisor $p$ of $N$. 
So, by appropriately ordering the prime divisors of the level, we formulate
this problem as follows:

\begin{ques}
Is there a newform $f$ of weight 2 and level $N=\prod\limits_{i=1}^t p_i$ with trivial character whose mod $\ell$ Galois representation is reducible such that $U_{p_i} f = f$ for $1 \leq i \leq s$ and $U_{p_j} f = -f$ for $ s+1 \leq  j \leq t$?
\end{ques}
We say a $t$-tuple $(p_1, \dots, p_t)$ of distinct primes \textit{admissible for $s$} if such a newform $f$ exists.
So, our question is to find admissible $t$-tuples for $s$ with $s\leq t$. 

\begin{thm}[Necessary conditions, Ribet]\label{thm:admnec}
Let $\ell\geq 5$ be a prime. Assume that a $t$-tuple $(p_1, \dots, p_t)$ is admissible for $s$. Then, we have
\begin{enumerate}
\item $s\geq 1$;
\item $\ell$ divides $\prod\limits_{i=1}^t (p_i-1)$ if $s=t$; 
\item $p_j \equiv -1 \modl$  for $s+1 \leq j \leq t$.
\end{enumerate}
\end{thm}

\begin{thm}[Sufficient conditions]\label{thm:sufficientmain}
Let $\ell\geq 5$ be a prime. Then, a $t$-tuple $(p_1, \dots, p_t)$ is admissible for $s$ if one of the following holds:
\begin{enumerate}
\item \label{th1} 
$s=t$ is odd and $\ell$ divides $\prod\limits_{i=1}^t (p_i-1)$;
\item \label{th2}
$s=t$ with even $t\geq 4$, and ``a complicated condition'' holds;
\item \label{th3}
$s+1=t$, $t$ is even, and $p_t \equiv -1 \modl$;
\item \label{th4}
$s=1$, $t\geq 2$, and $p_j \equiv -1 \modl$ for $2 \leq j \leq t$;
\item \label{th5} 
$s=2$, $t\geq 4$ is even, and $p_j \equiv -1 \modl$ for $3 \leq j \leq t$;
\item \label{th6} $s<t$, $t$ is even, $p_t \equiv -1 \modl$, and the $(t-1)$-tuple $(p_1, \dots, p_{t-1})$ is admissible for $s$;
\item \label{th7} $s<t$, $t$ is even, and the $(t-1)$-tuple $(p_2, \dots, p_t)$ is admissible for $(s-1)$.
\end{enumerate}
\end{thm}
\noindent (For the precise statement of the second case, see Theorem \ref{thm:s=t}; for the proof of the theorem, see Section \ref{sec:proof of 1.3}.) Note that in the (\ref{th6}) (resp. (\ref{th7})), we have $t-1 \geq s$ (resp. $s-1 \geq 1$) from our assumption on admissibility of the given $(t-1)$-tuples.

In most cases, the necessary conditions above are also sufficient for the admissibility. In particular, Theorem \ref{thm:sufficientmain} precisely characterizes admissible $t$-tuples for $s$ when $1\leq s \leq t\leq 4$, except in the cases $(s, t) \in \{(2, 2), ~(2, 3), ~(4, 4) \}$. In such exceptional cases, an extension of our methods allows us to obtain sufficient conditions, which are more complicated to state. We refer to Theorem \ref{thm:s2t2} (resp. Theorem \ref{thm:s2t3suff}, Theorem \ref{thm:s=t}) for the case $(s, t)=(2, 2)$ (resp. $(s, t)=(2, 3)$, $(s, t)=(4, 4)$).

In Theorem \ref{thm:sufficientmain} (\ref{th6}) and (\ref{th7}), we assume that $t$ is even because we use the geometry of Jacobians of Shimura curves in the proof. Therefore the discriminant of the quaternion algebra we consider has to be the product of an even number of distinct primes. In general, to study the case where $t$ is odd we need information about the kernel of the degeneracy map between Jacobians of Shimura curves (cf. \cite{R84}). Ribet conjectured the following.
\begin{conj}\label{conj:intro}
Let $D$ be the product of an even number of primes and let $p$ be a prime not dividing $D$.
Let $\gamma_p^{\mathrm{Sh}}$ denote the map
$$
\gamma_p^{\mathrm{Sh}} : J_0^D(1)\times J_0^D(1) \rightarrow J_0^D(p)
$$
induced by two degeneracy maps $\alpha_p$ and $\beta_p$. Then, the kernel of $\gamma_p^{\mathrm{Sh}}$ is an antidiagonal embedding of the Skorobogatov subgroup 
$$
\Sigma := \bigoplus_{\substack q \mid D}\Sigma_q
$$
of $J_0^D(1)$, where $\Sigma_q$ is the Skorobogatov subgroup of $J_0^D(1)$ at $q$ introduced in Appendix \ref{sec:Skorobogatov}.
\end{conj}
\noindent (See Section \ref{sec:notation} for any unfamiliar notation or terminology.) 

If this conjecture is true, then we can raise the level as above.

\begin{thm} \label{thm:1.5}
Let $\ell \geq 5$ be a prime. Assume that a $t$-tuple $(p_1, \dots, p_t)$ is admissible for $s$ and $t$ is even. 
Assume further that Conjecture \ref{conj:intro} holds. Then,
\begin{enumerate}
\item \label{thm:1.5.1}
a $(t+1)$-tuple $(p_1, \dots, p_t, p_{t+1})$ is admissible for $s$ if and only if $p_{t+1} \equiv -1 \modl$;

\item 
a $(t+1)$-tuple $(p_0, p_1, \dots, p_t)$ is admissible for $(s+1)$ if $s+2 \leq t$.
\end{enumerate}
\end{thm}  

These level-raising methods give us the following necessary and sufficient conditions for the admissibility.

\begin{thm}\label{thm:1.6}
Let $\ell \geq 5$ be a prime. Assume that Conjecture \ref{conj:intro} holds. Then, 
a $t$-tuple $(p_1, \dots, p_t)$ is admissible for $s$ when
\begin{enumerate}
\item 
$s\neq t$ with even $t$ if and only if $p_{s+1}\equiv \cdots \equiv p_t \equiv -1 \modl$;

\item 
$s+2 \leq t$ with odd $t$ if and only if $p_{s+1}\equiv \cdots \equiv p_t \equiv -1 \modl$.
\end{enumerate}
\end{thm}

Combining all the results above, the only missing cases are those with 
$s+1=t$ for $t\geq 5$ odd\footnote{In the cases $s=t$ with $t$ even and $(s, t)=(2, 3)$, we only have a mild sufficient condition, which may not be optimal.}.
We may obtain a sufficient condition for these cases using the techniques in the proof of Theorem \ref{thm:s2t3suff}, but we will not pursue this here. 
Instead, we discuss how Conjecture \ref{conj:intro} is used to remove an assumption in Theorem \ref{thm:s2t3suff} in Section \ref{sec: remaining cases}.
\vspace{3mm}

The organization of this article is as follows. 
In Section \ref{sec:ribetwork}, we introduce Ribet's work, which was announced in his CRM lecture \cite{R10}. 
In Section \ref{sec:levelraising}, we study level raising methods that are main tools of this article. 
In Section \ref{sec:proofs}, we present a complete proof of Ribet's work on admissible tuples using the results in the previous section. 
In Sections \ref{sec:admtuples} and \ref{sec:admtriples}, we discuss generalizations of Ribet's work and give some examples of admissible triples for $s=2$ and of admissible quadruples for $s=4$. 
In Section \ref{sec:conjecture}, we prove Theorems \ref{thm:1.5} and \ref{thm:1.6}.
Finally in the appendices, we provide some known results on arithmetic of Jacobian varieties of modular curves and Shimura curves. We include some proofs of them for reader's convenience.

\subsection*{Acknowledgements}
The author would like to thank his advisor Kenneth Ribet for his inspired suggestions and comments, which led him writing this paper. 
He is also grateful to Seunghwan Chang, Yeansu Kim, Sug Woo Shin, and Gabor Wiese for many suggestions toward the correction and improvement of this article; and to Chan-Ho Kim for providing examples in Section \ref{sec:admtriples}.
Finally, he would like to thank the referee for numerous
corrections, suggestions and valuable remarks.

\subsection{Notation}\label{sec:notation}
Let $B$ be an indefinite quaternion algebra over $\Q$ of discriminant $D$ with a fixed embedding $\phi : B \inj M_2(\R)$. 
(Hence, $D$ is the product of an even number of distinct primes.)
Let $\cO$ be an Eichler order of level $N$ of $B$, and set $\Gamma_0^D(N)=\cO^{\times, 1}$, the set of (reduced) norm 1 elements in $\cO$.
Let $X_0^D(N)$ be the Shimura curve for $B$ with $\Gamma_0^D(N)$ level structure. Let $J_0^D(N)$ be the Jacobian of $X_0^D(N)$. If $D=1$, $X_0(N)=X_0^1(N)$ denotes the modular curve for $\Gamma_0(N)$ and $J_0(N)=J_0^1(N)$ denotes its Jacobian variety. (Note that if $D\neq 1 $, $X_0^D(N)(\C) \simeq {\phi(\Gamma_0^D(N))} \backslash \fH$, where $\fH$ is the complex upper half plane.)

We can consider the N\'eron model of $J_0^D(N)$ over $\Z$, which is denoted by $J_0^D(N)_{/{\Z}}$ (cf. \cite{Bu97, Ce76, DR73, Dr76, Hm, Ig59, KM85, Ra70}). We denote by $J_0^D(N)_{/{\F_p}}$ the special fiber of $J_0^D(N)_{/{\Z}}$ over $\F_p$. (For instance, if $p$ exactly divides $N$ (resp. $D$), it is given by the Deligne-Rapoport (resp. Cerednik-Drinfeld) model and the theory of Raynaud. Note that we will only consider these cases.)
For a Jacobian variety $J$ over $\Q$, we denote by
$X_p(J)$ (resp. $\Phi_p(J)$) the character (resp. component) group of its special fiber $J_{/{\F_p}}$ of the N\'eron model $J_{/\Z}$. 

Let $T_n$ be the $n^\th$ Hecke operator acting on $J_0^D(N)$. We denote by $\T^D(N)$ the $\Z$-subalgebra of the endomorphism ring of $J_0^D(N)$ generated by
all $T_n$. In the case where $D=1$, we often denote $\T^1(N)$ by $\T(N)$ .
If $p$ divides $DN$, we often denote by $U_p$ the Hecke operator $T_p$ acting on $J_0^D(N)$. 
For a prime $p$ dividing $N$, we denote by $w_p$ the Atkin-Lehner involution acting on $J_0^D(N)$. 
For a maximal ideal $\m$ of a Hecke ring $\T$, we denote by $\T_{\m}$ the completion of $\T$ at $\m$, i.e.,
$$
\T_{\m}:= \lim_{\leftarrow n} \T/{\m^n}.
$$

There are two degeneracy maps $\alpha_p, \beta_p : X_0^D(Np) \rightarrow X_0^D(N)$ for a prime $p$ not dividing $DN$.
Here, $\alpha_p$ (resp. $\beta_p$) is the one induced by ``forgetting the level $p$ structure'' (resp. by ``dividing by the level $p$ structure''). 
For any divisor $M$ of $N$, we denote by $J_0^D(N)_{M\hyp\new}$ the $M$-new subvariety of $J_0^D(N)$ (cf. \cite{R84}).
We also denote by $\T^D(N)^{M\hyp\new}$ the image of $\T^D(N)$ in the endomorphism ring of $J_0^D(N)_{M\hyp\new}$. If $M=N$, we define $J_0^D(N)_{\new} :=J_0^D(N)_{N\hyp\new}$ and
$\T^D(N)^{\new}:=\T^D(N)^{N\hyp\new}.$ A maximal ideal of $\T^D(N)$ is called \textit{$M$-new} if its image in $\T^D(N)^{M\hyp\new}$ is still maximal; and it is called \textit{new} if it is $N$-new.
\vv

From now on, we always assume that $\ell\geq 5$ is a prime and $N$ is a square-free integer prime to $\ell$. For such an integer $N$, we define two arithmetic functions $\varphi(N)$ and $\psi(N)$ by
$$
\varphi(N):=\prod_{p\mid N ~\text{primes}} (p-1)\quad\text{and}\quad \psi(N):=\prod_{p\mid N ~\text{primes}} (p+1).
$$ 

Since we focus on Eisenstein maximal ideals of residue characteristic $\ell\geq 5$, 
we introduce the following notation for convenience.
\begin{notation}\label{notation}
We say that for two natural numbers $a$ and $b$, $a$ is equal to $b$ \textit{\upto} if $a=b \times 2^x 3^y$ for some integers $x$ and $y$. For two finite abelian groups $A$ and $B$, we denote by $A\sim B$ if $A_{\ell}:=A\otimes \Z_{\ell}$, the $\ell$-primary subgroup of $A$, is isomorphic to $B_{\ell}$ for all primes $\ell$ not dividing $6$.
\end{notation}

We denote by $\chi_{\ell}$ the mod $\ell$ cyclotomic character, i.e.,
$$
\chi_{\ell} : \Gal(\overline{\Q}/{\Q}) \twoheadrightarrow \Gal({\Q(\zeta_{\ell})}/{\Q}) \simeq (\Z/{\ell\Z})^{\times} \rightarrow \F_{\ell}^{\times},
$$
where $\zeta_{\ell}$ is a primitive $\ell^\th$  root of unity. Note that $\chi_{\ell}$ is unramified outside $\ell$ and $\chi_{\ell}(\Frob_p) \equiv p \modl$ for a prime $p \neq \ell$, where $\Frob_p$ denotes an arithmetic Frobenius element for $p$ in $\Gal(\overline{\Q}/{\Q})$.

For an ideal $\m$ of $\T$ and a variety $A$ over a field $K$ which is a $\T$-module, $A[\m]$ denotes the kernel of $\m$ on $A$, i.e.,
$$
A[\m] := \{x \in A(\overline{K}) ~:~ Tx = 0~\text{ for all }~ T \in \m  \}.
$$

Finally, we denote by $\cI_0^D(N)$ the (minimal) Eisenstein ideal in the Hecke ring $\T^D(N)$, i.e.,
$$
\cI_0^D(N) := (T_r-r-1~:~\text{ for primes } r \nmid DN) \subset \T^D(N).	
$$
When $D=1$, we often denote it by $\cI_0(N)$;
when we denote $\T^D(N)$ by $\T$, we also denote $\cI^D_0(N)$ by $\cI_0$.
\vv

\section{Ribet's work}\label{sec:ribetwork}
In this section, we discuss Ribet's results on reducible representations arising from modular forms of weight two for $\Gamma_0(N)$ with a square-free integer $N$. For their proofs, see Section \ref{sec:proofs}.

\subsection{Reducible mod $\ell$ Galois representations arising from newforms}
Let $\ell \geq 5$ be a prime and let $N$ be a square-free integer prime to $\ell$. 
Let $f$ be a newform of weight 2 for $\Gamma_0(N)$.
Assume that $\ov{\rho}_f$, the semisimplified mod $\ell$ Galois representation attached to $f$, is reducible. Then, we have the following.
\begin{prop}[Ribet]\label{prop:reducible}
$\ov{\rho}_f$ is isomorphic to $\mathbbm{1} \oplus \chi_{\ell}$.
\end{prop}

This is well-known, and even more is true: for a general weight $k$ with $2 \leq k \leq \ell+1$, we have $\overline{\rho}_f \simeq \mathbbm{1} \oplus \chi_{\ell}^{k-1}$. For a proof of this statement, see 
\cite[Proposition 3.1]{BM1}.

\subsection{Admissible tuples}
Fix a prime $\ell\geq 5$.
Fix $t$, the number of prime factors of $N$, and $s \in \{1, \dots, t\}$, the number of plus signs.
(By Theorem \ref{thm:admnec}, the case $s=0$ is excluded.)

We seek to characterize $t$-tuples $(p_1, \dots, p_t)$ of distinct primes so that there is a newform $f$ of level $N=\prod\limits_{i=1}^t p_i$ with weight two and trivial character such that
\begin{enumerate}
\item $\ov{\rho}_f \simeq \mathbbm{1} \oplus \chi_{\ell}$;
\item $U_{p_i} f = f$ for $1 \leq i \leq  s$;
\item $U_{p_j} f = -f$ for $s+1 \leq j \leq t$.
\end{enumerate}
We say these $t$-tuples \textit{admissible for $s$}. When we discuss admissible tuples, we always fix a prime $\ell\geq 5$ and assume that the level $N$ is prime to $\ell$. 

\subsection{Results on admissible tuples}
In this subsection, we introduce the work of Ribet on admissible tuples, which was announced in his CRM lecture \cite{R10}. 
For a proof, see Section \ref{sec:proofs}.

First, Ribet provided some necessary conditions for admissibility: see Theorem \ref{thm:admnec}.
A complete proof of it is already published: see \cite[Theorem 2.6 (ii)]{BD1}.
Let us remark that it can be generalized to the case where $\ell=3$. The first and last statements are valid without any change; however, the second statement can be improved as follows: ``If $s=t$, then $9 \mid \varphi(N)$''. For the proof of the first and second statements for $\ell=3$, see \cite[Theorem 1.4 and Proposition 5.5]{Yoo15a}. For the proof of the last statement for $\ell=3$, see the proof of \cite[Theorem 2.6(ii)(a)]{BD1}, which is valid without any change.
\vv

Now, assume that a $t$-tuple $(p_1, \dots, p_t)$ is admissible for $s$. (So, $1\leq s\leq t$.) By Theorem \ref{thm:admnec}, if $s=t$ then $\ell$ divides $\varphi(N)$; and if $s+1=t$ then $p_t \equiv -1 \modl$. Ribet proved that they are also sufficient when $s$ is odd.

\begin{thm}[Ribet] \label{thm:admsuff}
A $t$-tuple $(p_1, \dots, p_t)$ is admissible for $s$ if one of the following holds:
\begin{enumerate}
\item if $s=t$ and $s$ is odd, then $\ell \mid \prod\limits_{i=1}^t (p_i-1)$;
\item if $s+1=t$ and $s$ is odd, then $p_t \equiv -1 \modl$.
\end{enumerate}
\end{thm}

By the Theorems \ref{thm:admnec} and \ref{thm:admsuff}, a single $(p)$ is admissible for $s=1$ if and only if $p \equiv 1 \modl$; and a pair $(p, q)$ is admissible for $s=1$ if and only if $q \equiv -1 \modl$;
on the other hand, we only have a necessary condition for a pair $(p, q)$ to be admissible for $s=2$, which is $\ell \mid \varphi(pq)=(p-1)(q-1)$. Without loss of generality, we may assume that $p \equiv 1 \modl$. Then, we have the following.
\begin{thm}[Ribet] \label{thm:s2t2}
Suppose that $p\equiv 1 \modl$. Then, 
a pair $(p, q)$ is admissible for $s=2$ if and only if either one of the following holds:
\begin{itemize}
\item
$q \equiv 1 \modl$. 
\item
$q$ is an $\ell^\th$  power modulo $p$.
\end{itemize}
\end{thm}

When $t$ is even, Ribet proved the following level raising theorem.
\begin{thm}[Ribet] \label{thm:ribetraising}
Assume that a $(t-1)$-tuple $(p_1, \dots, p_{t-1})$ is admissible for $s$. Assume further that $t$ is even.
Then, a $t$-tuple $(p_1, \dots, p_t)$ is admissible for $s$ if and only if $p_t \equiv -1 \modl$.
\end{thm}
Here, $1\leq s \leq t-1$ by our definition on admissibility for $s$.
\vv

\section{Level raising methods} \label{sec:levelraising}
In his article \cite{R84}, Ribet studied the kernel of the map
$$
\gamma_p : J_0(N) \times J_0(N) \rightarrow J_0(Np)
$$
induced by two degeneracy maps. He also computed the intersection of the $p$-new subvariety and the $p$-old subvariety of $J_0(Np)$. Diamond and Taylor generalized Ribet's result \cite{DT94}, and they determined the non-optimal levels of irreducible mod $\ell$ modular representations by level raising methods. However we cannot directly use their methods to find non-optimal levels of $\mathbbm{1} \oplus \chi_{\ell}$. The reason is basically that the kernels of their level raising maps are ``Eisenstein''  and we do not know how $U_p$ acts  on the kernels for primes $p$ dividing the level. Instead, we introduce new level raising methods based on an exact sequence
due to Ribet. 

\subsection{Equivalent condition}
Let $\m$ be a maximal ideal of the Hecke ring $\T(N)$ of residue characteristic $\ell$. Let $\rho_{\m}$ be the semisimplified mod $\ell$ Galois representation attached to $\m$. Assume that $p$ is a prime not dividing $N$. We say \textit{level raising occurs for $\m$} (from level $N$ to level $Np$) 
if there is a maximal ideal 
$\fn$ of $\T({Np})$ such that
\begin{enumerate}
\item $\fn$ is $p$-new and
\item $\rho_{\fn}$, the semisimplified mod $\ell$ representation attached to $\fn$, is isomorphic to $\rho_{\m}$.
\end{enumerate}

Let $\T:=\T({Np})$. Note that a maximal ideal $\m$ of $\T(N)$ can be regarded as a maximal ideal $\wt{\m}$ of $\T^{p\hyp\old}$ once we replace an element $T_p- a_p \in \m$ by $U_p-\g$, where $\g$ is a root of the polynomial $X^2-a_p X+p \pmod \m$\footnote{Let $R$ be the common subring of $\T(N)$ and $\T^{p\hyp\old}$, generated by $T_r$ for primes $r \neq p$  (cf. \cite[\textsection 7]{R90}). 
(Two rings are considered as subrings of the endomorphism ring of $J_0(N)^2$.) 
Then, $\T(N)=R[T_p]$ and $\T^{p\hyp\old}=R[U_p]$.
Let $V=J_0(N)[\m]:=\{ x \in J_0(N)(\ov{\Q}) : T x = 0 \text{ for all } T \in \m \}$, which is non-zero, and consider $V^2$ as a submodule of $J_0(N)^2$. Then we can find a subspace $U$ of $V^2$ which is fixed by $U_p-\g$ and isomorphic to $V$. Let $a(r)$ be the image of $T_r$ in $\T(N)/\m$ for a prime $r$ 
(i.e. $\m:=(T_r-a(r) : \text{for all primes } r)$).
Then, $T_r$ acts on $U$ as $a(r)$, i.e., $U$ is stable under all the Hecke operators. Furthermore, $U_p$ acts on $U$ as $\g$. Therefore, as subrings of $\End(U)$ we have
\[
\T(N)/\m =R[T_p]/{(T_r-a(r) : \text{for all primes } r)}\simeq R[U_p]/{(U_p-\g, T_r-a(r) : \text{for all primes } r \neq p)}= \T^{p\hyp\old}/\wt{\m},
\]
where $\wt{\m}:=(U_p-\g, T_r-a(r) : \text{for all primes } r \neq p)$. 
For our applications, $a(p) \equiv 1+p \pmod \m$ and hence $\g$ is either $1$ or $p$. And $U=\{(x,-x) : x \in V\}$ or $\{(px, -x) : x \in V\}$ for $\g = 1$ or $p$, respectively.}.
By abusing notation, let $\m$ be a maximal ideal of $\T$ whose image in $\T^{p\hyp\old}$ is $\wt{\m}$. If level raising occurs for $\m$, $\m$ is also $p$-new; in other words, the image of $\m$ in $\T^{p\hyp\new}$ is also maximal. 

To raise the level, Ribet showed that all congruences between $p$-new and $p$-old forms can be detected geometrically by the intersection between the $p$-old and the $p$-new parts of the relevant Jacobian as follows.
\begin{thm}[Ribet]\label{thm:geometricintersection}
Let $J:=J_0(Np)$. As before, assume that $Np$ is prime to $\ell$ and $p \nmid N$. Let $\m$ be a maximal ideal of $\T$ of residue characteristic $\ell$, which is $p$-old. Then level raising occurs for $\m$ if and only if
$$
J_{p\hyp\old} \bigcap J_{p\hyp\new}[\m] \neq 0.
$$
\end{thm}
\begin{proof}
Let $\Omega:=J_{p\hyp\old} \bigcap J_{p\hyp\new}$. If $\Omega[\m] \neq 0$, then $J_{p\hyp\new}[\m]$ is not zero, and hence $\m$ is $p$-new. 

Conversely, assume that $\Omega[\m] = 0$. Consider the following exact sequence:
$$
\xymatrix{
0 \ar[r] & \Omega \ar[r] & J_{p\hyp\old} \times J_{p\hyp\new} \ar[r] & J \ar[r] & 0.
}
$$
Let $e = (1, 0) \in \End(J_{p\hyp\old}) \times \End(J_{p\hyp\new})$. 
Note that $e \in \T \otimes_{\Z} \Q$ because $\Omega$ is finite. 
If $e \not\in \End(J)$, then $J \not\simeq J_{p\hyp\old} \times J_{p\hyp\new}$. Therefore $e \in \End(J) \otimes_{\T} \T_{\m}$ because $\Omega[\m]=0$. Thus, we have
$$
e \in (\T {\otimes}_{\Z} \Q) \bigcap (\End(J) \otimes_{\T} {\T_{\m}}).
$$
Note that the intersection $(\T {\otimes}_{\Z} \Q) \bigcap (\End(J)\otimes_{\T} {\T_{\m}})$ is equal to $\T'_\m$, where $\T'$ is the saturation\footnote{The saturation of $\T$ in $\End(J)$ is the intersection of $\T \otimes_\Z \Q$ and $\End(J)$ \cite[p. 29]{ARS12}.} of $\T$ in $\End(J)$. Since $\T$ is saturated in $\End(J)$ locally at $\m$ by \cite[Proposition 5.10]{ARS12}, i.e., 
\[
\T_\m \simeq \T'_\m,
\]
we have $e \in \T_{\m}$.

If $\m$ is also a maximal ideal after projection $\T \rightarrow \T^{p\hyp\new}$, the injection $\T \hookrightarrow
\T^{p\hyp\old} \times \T^{p\hyp\new}$ is not an isomorphism after taking completions at $\m$. Thus, $e = (1, 0) \in \T^{p\hyp\old} \times \T^{p\hyp\new}$ cannot belong to $\T_{\m}$, which is a contradiction. Therefore $\m$ is not $p$-new.
\end{proof}

\begin{rem}
The following assertion is equivalent to \cite[Proposition 5.10]{ARS12}.
\begin{thm*}[Agashe, Ribet, and Stein]
Let $\ell$ be the characteristic of $\T(N)/{\m}$. Then, $\T(N)$ is saturated in $\End(J_0(N))$ locally at $\m$ if
\begin{enumerate}
\item $\ell \nmid N$, or
\item $\ell ~\|~N$ and $T_{\ell} \equiv \pm 1~ \pmod \m$.
\end{enumerate}
\end{thm*}
In our case, the level $Np$ is prime to $\ell$ and hence $\T({Np})$ is saturated in $\End(J_0(Np))$ locally at $\m$.
\end{rem}

In the proof of the theorem above, they used the $q$-expansion principle of modular forms.
In general, when we consider Jacobians of Shimura curves, the saturation property of the Hecke algebra is difficult to prove because we don't have the $q$-expansion principle. However, we can prove the following.

\begin{prop}
Let $\T:=\T^{pr}(q)$, $J:=J_0^{pr}(q)$, and $\m:=(\ell, ~U_p-1, ~U_q-1, ~ U_r+1, ~\cI_0) \subseteq \T$. Assume that $\ell$ does not divide $(p-1)(q-1)$ and $r \equiv -1 \modl$. Then, $\T$ is saturated in $\End(J)$ locally at $\m$.  
\end{prop}
\begin{proof}
It suffices to find a free $\T_{\m}$-module of finite rank on which $\End(J)$ operates by functoriality as in the article \cite{ARS12}.

Let $Y$ (resp. $L$, $X$) be the character group of $J_0^{pr}(q)$ at $r$ (resp. $J_0(pqr)$, $J_0(pq)$ at $p$). By Ribet \cite{R90}, there is an exact sequence:
$$
\xymatrix{
0 \ar[r] & Y \ar[r] & L \ar[r] & X \oplus X \ar[r] & 0.
}
$$
Let $\fa$ (resp. $\fb$) be the maximal ideal corresponding to $\m$ in $\T({pqr})$ (resp. $\T({pq})$).
Since $p \not\equiv 1\modl$ and a pair $(p, q)$ is not admissible for $s=2$, $\fb$ is not $p$-new. Therefore $X \oplus X$ does not have support at $\fb$. (Note that the action of $\T({pq})$ on $X$ factors through $\T({pq})^{p\hyp\new}$.) Thus, we have
$Y_{\m} \simeq L_{\fa}$. Since $L/{\fa L}$ is of dimension 1 over $\T(pqr)/{\fa}$ by Theorem \ref{thm:multipqr},
and $L$ is of rank 1 over $\T({pqr})$ in the sense of Mazur \cite[\ts II. 8]{M77} 
(cf. \cite[Lemma 4.13]{Hm}), $L_{\fa}$ is free of rank 1 over $\T({pqr})_{\fa}$ by Nakayama's lemma. 
Therefore $Y_{\m} \simeq L_{\fa}$ is also free of rank 1 over $\T_{\m}$. 
\end{proof}

\begin{rem}
Ribet provided the idea of the above proof.
\end{rem}

Using the proposition above and the proof of Theorem \ref{thm:geometricintersection}, we can deduce the following theorem easily.

\begin{thm} \label{thm:shimintersection}
Let $\T:=\T^{pr}(q)$, $J:=J_0^{pr}(q)$, and $\m:=(\ell, ~U_p-1, ~U_q-1, ~ U_r+1, ~\cI_0) \subseteq \T$. Assume that $\ell$ does not divide $(p-1)(q-1)$ and $r \equiv -1\modl$. Then, level raising occurs for $\m$ if and only if
$$
J_{q\hyp\old} \bigcap J_{q\hyp\new}[\m] \neq 0.
$$
\end{thm}
\vv

\subsection{The intersection of the $p$-old subvariety and the $p$-new subvariety}\label{sec:intersection}
As in the previous subsection, let $p$ be a prime not dividing $N$ and $\Omega$ the intersection of the $p$-old subvariety and the $p$-new subvariety of $J_0(Np)$. By two degeneracy maps, we have the following maps:
$$
\xymatrix{
J_0(N) \times J_0(N) \ar[r]^-{\gamma_p} & J_0(Np) \ar[r] & J_0(N) \times J_0(N).
}
$$
The composition is the matrix
$$
\delta_p : = \left(\begin{array}{cc}
p+1 & T_p \\
T_p & p+1
\end{array}\right).
$$
Let $\Delta$ be the kernel of $\delta_p$, i.e.,
$$
\Delta := J_0(N)^2[\delta_p] = \{ (x, y) \in J_0(N)^2 ~:~ (p+1)x = -T_p\, y ~\text{ and }~ T_p\, x = -(p+1)y \}.
$$
Note that since $\delta_p$ is an isogeny, the cardinality of $\Delta$ is finite.
Let $\Sigma$ be the kernel of $\gamma_p$. Then $\Delta$ contains $\Sigma$ and is endowed with a canonical non-degenerate alternating $\Gm$-valued pairing. Let $\Sigma^{\perp}$ be the orthogonal complement to $\Sigma$ relative to this pairing. Then, $\Sigma^{\perp}$ contains $\Sigma$ and we have the formula:
$$
\Omega = \Sigma^{\perp}/{\Sigma}.
$$
For more details, see \cite{R84}.

We define $\Delta^+$ and $\Delta^-$ as follows: 
\begin{align*}
\Delta^+ := \{(x, -x) \in J_0(N)^2 ~:~ x \in J_0(N)[T_p-p-1] \}, \\
\Delta^- := \{(x,~ x) \in J_0(N)^2 ~:~ x \in J_0(N)[T_p+p+1] \}.
\end{align*}
They are eigenspaces of $\Delta$ for the Atkin-Lehner operator $w_p$.
(Note that $w_p$ acts on $J_0(N)^2$ by swapping its components.) 
If we ignore 2-primary subgroups, $\Delta$ and $\Delta^+ \oplus \Delta^-$ coincide, i.e., $\Delta \sim \Delta^+\oplus \Delta^-$. (See Notation \ref{notation} for $\sim$.) Furthermore we have two filtrations as follows:
\begin{align*}
0 ~\subset~ \Sigma^+ ~\subset ~(\Sigma^{\perp})^+ ~\subset~ \Delta^+, \\
0 ~\subset~ \Sigma^- ~\subset ~(\Sigma^{\perp})^- ~\subset~ \Delta^-. 
\end{align*}
Since $\Delta/{\Sigma^{\perp}}$ is the $\Gm$-dual of $\Sigma$ and $\Sigma$ is an antidiagonal embedding of the
Shimura subgroup of $J_0(N)$ by Ribet \cite{R84}, we have $\Sigma^+ \sim \Sigma$ and $\Sigma^- \sim 0$. Therefore we have
$$
(\Sigma^{\perp})^- \sim \Delta^-.
$$
The function $\gamma_p$ maps $(\Sigma^{\perp})^+$ (resp. $(\Sigma^{\perp})^-$) to $(\Sigma^{\perp})^+/{\Sigma}$ (resp. $\Delta^-$) up to 2-primary subgroups. Since $\Sigma^{\perp}/{\Sigma}$ lies in $J_0(Np)_{p\hyp\new}$, it is annihilated by $U_p+w_p$. Hence, 
$(\Sigma^{\perp})^+/{\Sigma}$ (resp. $\Delta^-$) corresponds to the 
subspace of $\Omega$ annihilated by $U_p-1$ (resp. $U_p+1$) up to 2-primary subgroups.
\vv

\subsection{Ribet's exact sequence}\label{sec:ribetexactsequence}
Let $X_p(J)$ (resp. $\Phi_p(J)$) be the character (resp. component) group of $J$ at $p$.
By the degeneracy maps, there is a Hecke equivariant map between component groups:
$$
\Phi_p(J_0^D(Np)) \times \Phi_p(J_0^D(Np)) \rightarrow \Phi_p(J_0^D(Npq)),
$$
where $q$ is a prime not dividing $NDp$. Let $K$ (resp. $C$) be the kernel (resp. cokernel) of the map above. 
We recall \cite[Theorem 4.3]{R90} and its generalization by Helm (cf. \cite[Theorem 5.4]{Hm}).

\begin{thm}[Ribet]\label{thm:Ribetexactsequence}
There is a Hecke equivariant exact sequence:
$$
\xymatrix{
0 \ar[r] & K \ar[r] & X\oplus X/{\delta_q (X\oplus X)} \ar[r] & \Psi \ar[r] & C \ar[r] & 0,
}
$$
where
$$
X := X_p(J_0^D(Np)),~~\Psi:=\Phi_q(J_0^{Dpq}(N)),~~\text{and}~~~~
\delta_q :=\left( \begin{array}{cc}
q+1 & T_q \\
T_q & q+1 \end{array}
\right).
$$
\end{thm}
If we ignore 2-, 3-primary subgroups of $K$ (resp. $C$), it is isomorphic to $\Phi_p(J_0^D(Np))$ (resp. $\Z/{(q+1)\Z}$) (Remark \ref{rem:C C'}).
For a prime $\ell\geq 5$, we denote by $A_{\ell}$ the $\ell$-primary subgroup $A\otimes_{\Z} {\Z_{\ell}}$. 
We can decompose the exact sequence above into the eigenspaces by the action of the $U_q$ operator as follows:
\begin{cor} \label{cor:Ribetexactsequence} Let $\ell \geq 5$. Then, the following sequences are exact:
\begin{enumerate}
\item
$
\xymatrix{
0 \ar[r] & \Phi_p(J_0^D(Np))_{\ell} \ar[r] & (X/{(T_q-q-1)X})_{\ell} \ar[r] & \Psi^+_{\ell} \ar[r] & 0;
}
$
\item
$
\xymatrix{
0 \ar[r] & (X/{(T_q+q+1)X})_{\ell} \ar[r] & \Psi^-_{\ell} \ar[r] & C_\ell \ar[r] & 0,
}
$
\end{enumerate}
where $\Psi^+$ (resp. $\Psi^-$) denotes the subspace of $\Psi$ annihilated by $U_p-1$ (resp. $U_p+1$). 
\end{cor}
\begin{proof}
Since $K \sim \{(x, -x) : x \in \Phi_p(J_0^D(Np)) \}$ and $\Phi_p(J_0^D(Np))$ is annihilated by $T_q-q-1$, $U_q$ acts on $K_\ell$ by $1$. 
On the other hand, $U_q$ acts by $-1$ on $C$ by Corollary \ref{cor:cokernelcomp}. Therefore the result follows.
\end{proof}

\section{Proof of Ribet's work}\label{sec:proofs}
Even though Ribet's work has been explained in many lectures (e.g., \cite{R10}), 
a complete proof has not been published yet. In this section, we provide it based on his idea.
  
By Mazur's approach, using ideals of the Hecke algebra $\T(N)$, proving admissibility for $s$ of a $t$-tuple $(p_1, \dots, p_t)$ is equivalent to showing that the ideal
$$\m=(\ell, ~U_{p_i}-1, ~U_{p_j}+1, ~\cI_0(N) ~:~ 1 \leq i \leq s \text{ and } ~s+1 \leq j \leq t)$$
is a new maximal ideal. 
To prove that $\m$ is new, we seek a $\T^N(1)$-module (or $\T^{N/p}(p)^{p\hyp\new}$-module) $A$ such that $A[\fn]\neq 0$ (or $A/{\fn A} \neq 0$), where $\fn$ is the  ideal corresponding to $\m$ of $\T^N(1)$ (or $\T^{N/p}(p)^{p\hyp\new}$) by the Jacquet-Langlands correspondence. 
From now on, by abuse of notation we use the same letter for the corresponding ideals of $\T({DN})^{DM\hyp\new}$ and $\T^D(N)^{M\hyp\new}$ by the Jacquet-Langlands correspondence, where $M$ is a divisor of $N$.

\begin{proof} [Proof of Theorem \ref{thm:admsuff}] $~$
\begin{enumerate}
\item
Let $\m:=(\ell, ~U_{p_i}-1, ~\cI_0(N) ~:~ 1 \leq i \leq t) \subseteq \T(N)$. It is enough to show that $\m$ is new maximal. 

Assume that $\ell \mid \varphi(N)$. 
Let $p=p_1$ and $D=N/p$. Let $\Phi_p:=\Phi_p(J_0^D(p))$ be the component group of $J_0^D(p)_{/{\F_{p}}}$. Since $\ell \mid \varphi(N)$, we have $\Phi_p[\m] \neq 0$ by Propositions \ref{prop:heckecomp} and \ref{prop:ordercomp}. 
For $D=1$, by \cite[Theorem 3.10]{R90} and the monodromy exact sequence (\ref{eqn:monodromy}), the action of the Hecke ring on $\Phi_p$ factors through $\T^D(p)^{p\hyp\new}$, 
so $\m$ is $p$-new in $\T^D(p)$. By the Jacquet-Langlands correspondence, $\m$ is new. 
For $D\neq 1$, the same argument holds without further difficulties\footnote{The proof of \cite[Theorem 3.10]{R90} relies on the structure of $J_0(N)_{\F_p}$ when $p$ exactly divides $N$, which is obtained from
the Deligne-Rapoport model of $X_0(N)_{\F_p}$ and the theory of Raynaud. The same method works for $D\neq 1$ since we have the `generalized' Deligne-Rapoport model of $X_0^D(N)_{\F_p}$ by Buzzard and Helm. For more details of the structure of $J_0^D(N)_{\F_p}$, see Appendix \ref{appendix}.}.

\item
Let $q:=p_t$ and assume that $q \equiv -1 \modl$.
Let 
$\fn:=(\ell, ~U_{p_i}-1, ~U_q+1, ~\cI_0(N) ~:~ 1 \leq i \leq s)$ be an Eisenstein maximal ideal of $\T(N)$.
Let $p:=p_1$ and $D := N/{pq}$ (if $s=1$, set $D=1$). 
Since the number of distinct prime divisors of $D$ is even, there are Shimura curves $X_0^D(p), X_0^D(pq)$, and $X_0^{N}(1)$. 
By the Ribet's exact sequence in Corollary \ref{cor:Ribetexactsequence}, we have
$$
\xymatrix{
\Phi_q(J_0^{N}(1))_\ell \ar[r] & C_\ell \ar[r] & 0,
}
$$
where $C$ is the cokernel of the map $\Phi_p(J_0^D(p)) \times \Phi_p(J_0^D(p))\overset{\gamma_q}\ra  \Phi_p(J_0^D(pq))$.
By Corollary \ref{cor:cokernelcomp}, $C[\fn] \neq 0$. Therefore, $\fn$ is a proper maximal ideal of $\T^{Dpq}(1)$. By the Jacquet-Langlands correspondence, $\fn$ is a new maximal ideal of $\T(N)$; in other words, the given $t$-tuple is admissible for $s$.
\end{enumerate}
\end{proof}

\begin{proof}[Proof of Theorem \ref{thm:s2t2}] 
Since $p\equiv 1 \modl$, there is an Eisenstein maximal ideal of $\T:=\T(p)$ containing $\ell$.
Let $I:=(U_p-1, ~\cI_0) \subset \T$ and $\m := (\ell, ~I)$.
By Mazur \cite[Proposition 16.6]{M77}, $I_{\m}:=I \otimes \T_{\m}$ is a principal ideal and it is generated by $\eta_q:=T_q-q-1$ for a good prime $q$. (Here, a prime number $q$ different from $p$ is called \textit{good} if 
$q \not\equiv 1 \modl$ and $q$ is not an $\ell^\th$  power modulo $p$, cf. \cite[p. 124]{M77}.) In other words, the hypothesis on $q$ ensures that $\eta_q$ is not a generator of $I_\m$.

Assume that $\eta_q$ is not a generator of $I_{\m}$. By the Ribet's exact sequence in Corollary \ref{cor:Ribetexactsequence}, we have
$$
\xymatrix{
0 \ar[r] & \Phi_{\ell} \ar[r] & (X/{\eta_q X})_{\ell} \ar[r] & (\Psi^+)_{\ell} \ar[r] & 0,
}
$$
where $\Phi:=\Phi_p(J_0(p))$, $X:=X_p(J_0(p))$, and $\Psi:=\Phi_q(J_0^{pq}(1))$.
By Mazur\footnote{this result is hidden in his paper \cite{M77}. Since $\Phi$ is a cyclic module annihilated by $I$, it is a $\T/I$-module of rank at most 1. Since the orders of $\Phi$ and $\T/I$ are both the numerator of $\frac{p-1}{12}$, $\Phi$ is a free module of rank 1 over $\T/I$.}, $\Phi$ is a free module of rank 1 over $\T/I$ and by Ribet \cite[Theorem 2.3]{R88}, $X_{\m}$ is free of rank 1 over $\T_{\m}$. Since $(X/{I X}) \otimes_{\T} \T_{\m} \simeq X_{\m}/{I_{\m}X_{\m}} \simeq (\T/I)_{\m}$, if $\eta_q$ is not a generator of $I_{\m}$, then 
$$
\#(X/{\eta_q X})_{\m} > \#(X/{I X})_{\m} = \#(\T/I)_{\m} = \# \Phi_{\m}.
$$ 
In other words, 
after taking the completions of the exact sequence above at $\m$, we have $\Psi^+_{\fn} \neq 0$, where $\fn = (\ell, ~U_p-1, ~U_q-1, ~\cI_0^{pq}(1))$ is the ideal of $\T^{pq}(1)$ corresponding to $\m$. Thus, $\fn$ is maximal. By the Jacquet-Langlands correspondence, $\fn$ is new.

Conversely, assume that $\eta_q$ is a generator of $I_{\m}$. Let $\Omega$ be the intersection of the $q$-old subvariety and
the $q$-new subvariety of $J_0(pq)$. Let $\Delta :=J_0(p)^2[\delta_q]$ and let $\Sigma$ be the kernel of $\gamma_q$ as in Section \ref{sec:intersection}. We have a filtration of $\Delta^+$:
$$
0 ~\subseteq~ \Sigma ~\subseteq ~(\Sigma^{\perp})^+ ~\subseteq~ \Delta^+
$$
and $\Delta^+$ is isomorphic to $J_0(p)[\eta_q]$ up to 2-primary subgroups. Since $\eta_q$ is a generator of $I_{\m}$, $(\Delta^+)_{\m}$ is 
isomorphic to $J_0(p)[I]_{\m}$. By Mazur, $J_0(p)[I]$ is free of rank 2 (up to 2-primary part) and $\Sigma$ is free of rank 1 over
$\T/I$. Thus, the $\m$-primary subgroup of $(\Sigma^{\perp})^+/{\Sigma}$ is 0 because $\Delta^+/(\Sigma^{\perp})^+$ is the $\Gm$-dual of $\Sigma$ (and $\T/I$ is finite); in other words, the $\m$-primary subgroup of $\Omega$ is 0 and $\m$ is not in the support of $\Omega$. By Theorem \ref{thm:geometricintersection}, $\m$ is not $q$-new and hence the pair $(p, ~q)$ is not admissible for $s=2$.
\end{proof}

\begin{proof}[Proof of Theorem \ref{thm:ribetraising}] 
Assume that a $(t-1)$-tuple $(p_1, \dots, p_{t-1})$ is admissible for $s$ and $t$ is even.
Let $p=p_1$, $D=\prod_{j=2}^{t-1} p_j$, and $q=p_t$.  
(If $t=2$, set $D=1$.) Since the number of prime factors of $D$ is even, there are Shimura curves $X_0^{D}(p)$, $X_0^{D}(pq)$, and $X_0^{Dpq}(1)$. 

If a $t$-tuple $(p_1, \dots, p_t)$ is admissible for $s$, then $p_t=q \equiv -1 \pmod \ell$ by Theorem \ref{thm:admnec}.

Conversely, assume that $q \equiv -1 \pmod \ell$. 
Since the $(t-1)$-tuple $(p_1, \dots, p_{t-1})$ is admissible for $s$, there is a new Eisenstein maximal ideal 
$\m:=(\ell, ~U_{p_i}-1, ~U_{p_j}+1, ~\cI_0(pD) ~:~ 1 \leq i \leq s \text{ and } ~s+1 \leq  j \leq t-1)$ in $\T({pD})$. Let $X:=X_p(J_0^D(p))$ be the character group of $J_0^{D}(p)_{/{\F_p}}$. Then, by the Ribet's exact sequence in Corollary \ref{cor:Ribetexactsequence}, we have
$$
\xymatrix{
0 \ar[r] & (X/{(T_q+q+1)X})_{\ell} \ar[r] & (\Psi^-)_{\ell},
}
$$
where $\Psi:=\Phi_q(J_0^{Dpq}(1))$. Because $q \equiv -1 \pmod \ell$, $\ell \in \m$, and $T_q-q-1 \in \m$, we have $T_q+q+1 \in \m$. By the Jacquet-Langlands correspondence and the fact that $\T^D(p)^{p\hyp\new}$ acts faithfully on $X$, we have $X/{\m X} \neq 0$. Therefore 
$(X/{(T_q+q+1)X})_{\ell}$ has support at $\m$, so $\Psi^-[\fn] \neq 0$,
where 
$$\fn:=(\ell, ~U_{p_i}-1, ~U_{p_j}+1, ~\cI_0^{Dpq}(1) ~:~ 1 \leq i \leq s \text{ and } ~s+1 \leq j \leq t) \subseteq \T^{Dpq}(1).$$ Therefore $\fn$ is maximal. By the Jacquet-Langlands correspondence, $\fn$ is new. 
\end{proof}

\section{Admissible tuples for $s=1$ or even $t$} \label{sec:admtuples}
In this section we present new results on admissible tuples for $s=1$ or even $t$. As always, we assume that $\ell \geq 5$.

\begin{thm} \label{thm:adms1}
Assume that $t\geq 2$.
A $t$-tuple $(p_1, \dots, p_t)$ is admissible for $s=1$ if and only if $p_i \equiv -1 \modl$ for $2 \leq i \leq t$.
\end{thm} 

\begin{proof} 
Assume that a $t$-tuple $(p_1, \dots, p_t)$ is admissible for $s=1$ . Then, by Theorem \ref{thm:admnec}, we have 
$p_i \equiv -1 \modl$ for $2 \leq i \leq t$. 

Conversely, assume that $p_i \equiv -1 \modl$ for $2 \leq i \leq t$. 
\begin{enumerate}
\item Case 1 : 
Assume that $t$ is odd and a $(t-1)$-tuple $(p_1, \dots, p_{t-1})$ for $s=1$ is admissible.  
Let $p:=p_1$, $q:=p_t$, $C=\prod_{i=3}^{t-1} p_i$, and $D=\prod_{i=2}^{t-1} p_i$. (If $t=3$, set $C=1$.) 
Let $$\m:=(\ell, ~U_p-1, ~U_{p_i}+1, ~\cI_0(pD) ~:~  2 \leq i \leq {t-1})$$ 
be a new Eisenstein 
maximal ideal of $\T({pD})$. Let 
$$\fn:=(\ell, ~U_p-1, ~U_{p_i}+1, ~\cI_0(pDq) ~:~ 2 \leq i \leq t)$$ 
be an Eisenstein maximal ideal of $\T({pDq})$. It suffices to show that $\fn$ is new. 

Since $t$ is odd, there are Shimura curves $X_0^C(p p_2)$ and $X_0^{Dq}(p)$. By the Ribet's exact sequence in Corollary \ref{cor:Ribetexactsequence},
we have
$$
\xymatrix{
0 \ar[r] & (X/{(T_q+q+1)X})_{\ell} \ar[r] & \Psi^-_{\ell} \ar[r] & (\Z/{(q+1)\Z})_{\ell} \ar[r] & 0, 
}
$$
where $X:=X_{p_2}(J_0^C(p p_2))$ and $\Psi:=\Phi_q(J_0^{Dq}(p))$. Since $\T^C({p p_2})^{p_2 \hyp \new}$ acts faithfully on $X$ and $\m$ is new, we have $X/{\m X} \neq 0$. Moreover, $T_q+q+1 \in \m$ because $q \equiv -1 \modl$, $\ell \in \m$, and $T_q-q-1 \in \m$. Thus, $\m$ is in the support of $(X/{(T_q+q+1)X})_{\ell}$, so $\fn$ is also in the support of $\Psi^-_{\ell}$.
Therefore $\fn$ is a proper maximal ideal of $\T^{Dq}(p)$.  
If $\fn$ is $p$-old, then
there is a new maximal ideal $(\ell, ~U_{p_i}+1, ~\cI_0(Dq) ~:~ 2 \leq i \leq t)$ of $\T({Dq})$ 
by the Jacquet-Langlands correspondence; in other words, the $(t-1)$-tuple $(p_2, \dots, p_t)$ is admissible for $s=0$, which contradicts Theorem \ref{thm:admnec}. Thus, $\fn$ is $p$-new, so by the Jacquet-Langlands
correspondence $\fn$ is a new maximal ideal of $\T({pDq})$.

\item Case 2 : 
Assume that $t$ is even and a $(t-1)$-tuple $(p_1, \dots, p_{t-1})$ is admissible for $s=1$.
Then, since $p_t \equiv -1 \modl$, the $t$-tuple $(p_1, \dots, p_t)$ for is admissible for $s=1$ by Theorem \ref{thm:ribetraising}.
\end{enumerate}

When $t=2$, by Theorem \ref{thm:admsuff}, the pair $(p_1, p_2)$ is admissible for $s=1$. Thus, by induction on $t$, the $t$-tuple $(p_1, \dots, p_t)$ is admissible for $s=1$ for all $t \geq 2$. 
\end{proof}

Using the same method as above, we can prove the following level raising theorem, which is almost complement of the case in Theorem \ref{thm:ribetraising} when $t$ is even. (This excludes the case where $s=t$ only.)
  
\begin{thm}\label{thm:yooraising}
Assume that $t$ is even and $s<t$. And assume that a $(t-1)$-tuple $(p_2, \dots, p_t)$ is admissible for $(s-1)$. Then, a $t$-tuple $(p_1, \dots, p_t)$ is admissible for $s$. 
\end{thm} 
In contrast to Theorem \ref{thm:ribetraising}, there is no condition on $p_1$ for raising the level. 
\begin{proof}
Let $q=p_1$ and $\m:=(\ell, ~U_{p_i}-1, ~U_{p_j}+1, ~\cI_0(pD) ~:~ 2 \leq i \leq s \text{ and }~ s+1 \leq j \leq t)$ be an Eisenstein maximal ideal of $\T({pD})$, where $D:=\prod_{k=2}^{t-1} p_k$ and $p:=p_t$. By our assumption, $\m$ is new. 
Let $\fn:=(\ell, ~U_{p_i}-1, ~U_{p_j}+1, ~\cI_0(Dpq) ~:~ 1 \leq i \leq s \text{ and } ~ s+1 \leq j \leq t)$ be an Eisenstein maximal ideal of $\T({Dpq})$. It suffices to show that $\fn$ is new. 

Since $t$ is even, there are Shimura curves $X_0^D(p)$ and $X_0^{Dpq}(1)$. 
By the Ribet's exact sequence in Corollary \ref{cor:Ribetexactsequence}, we have
$$
\xymatrix{
0 \ar[r] & \Phi_{\ell} \ar[r] & (X/{(T_q-q-1)X})_{\ell} \ar[r] & \Psi^+_{\ell} \ar[r] & 0,
}
$$
where $\Phi:=\Phi_p(J_0^D(p))$, $X:=X_p(J_0^D(p))$, and $\Psi:=\Phi_q(J_0^{Dpq}(1))$.
Since $\m$ is new and $\T^D(p)^{p\hyp \new}$ acts faithfully on $X$, we have $X/{\m X} \neq 0$.
Because $T_q-q-1 \in \m$, $\m$ lies in the support of $(X/(T_q-q-1)X)_{\ell}$. Since   
$\m$ is not in the support of $\Phi_{\ell}$ by Proposition \ref{prop:heckecomp}, 
$\fn$ is in the support of $\Psi^+_{\ell}$. Thus, $\fn$ is a maximal ideal of $\T^{Dpq}(1)$. By the Jacquet-Langlands
correspondence, $\fn$ is new.
\end{proof}

\begin{cor}\label{cor:adms2}
Assume that $t\geq 4$ is even. Then, a $t$-tuple $(p_1, \dots, p_t)$ is admissible for $s=2$ if and only if
$p_i \equiv -1 \modl$ for $3 \leq i \leq t$.
\end{cor}
\begin{proof}
If a $t$-tuple $(p_1, \dots, p_t)$ is admissible for $s=2$, then $p_i \equiv -1 \modl$ for $3 \leq i \leq t$ by Theorem \ref{thm:admnec}.

Conversely, assume that $p_i \equiv -1 \modl$ for $3 \leq i \leq t$. 
By Theorem \ref{thm:adms1}, the $(t-1)$-tuple $(p_2, \dots, p_t)$ is admissible for $s=1$. Thus,
the $t$-tuple $(p_1, \dots, p_t)$ is also admissible for $s=2$ by Theorem \ref{thm:yooraising}. 
\end{proof}
\vv

\section{Admissible triples and quadruples}\label{sec:admtriples}
In this section, we classify admissible triples and quadruples.  As before, we assume that $\ell \geq 5$. In the last subsection, we provide the proof of Theorem \ref{thm:sufficientmain} for the convenience of readers.
\subsection{Admissible triples}
By Theorem \ref{thm:admnec} and \ref{thm:admsuff}, a triple $(p, q, r)$ is admissible  for $s=3$ if and only if $\ell \mid \varphi(pqr)$; and by Theorem \ref{thm:adms1}, a triple $(p, q, r)$ is admissible for $s=1$ if and only if $q\equiv r\equiv -1 \modl$.

However, we cannot directly use the above theorems to get admissible triples for $s=2$.
By Theorem \ref{thm:admnec}, if a triple $(p, q, r)$ is admissible for $s=2$, then $r \equiv -1 \modl$. 
Assume that $r \equiv -1 \modl$. Let $I:=(U_p-1, U_r+1, ~\cI_0(pr))$ be an Eisenstein ideal 
of $\T:=\T({pr})$ and $\m := (\ell, ~I)$. Since $r \equiv -1 \modl$, $\m$ is new by Theorem \ref{thm:admsuff}. We discuss sufficient conditions for a triple $(p, q, r)$ to be admissible for $s=2$ by level raising methods.

\begin{thm} \label{thm:s2t3suff}
Assume that $p \not\equiv 1 \modl$ and $r\equiv -1 \modl$; and if $q \equiv 1 \modl$, we assume further that $p$ is not an $\ell^\th$  power modulo $q$.
Then, a triple $(p, q, r)$ is admissible for $s=2$ if $\eta_q:=T_q-q-1$ is not a generator of $I_{\m}$.

Assume further that $q \not\equiv 1 \modl$ and $r \not\equiv -1 \pmod {\ell^2}$. Then, a triple $(p, q, r)$ is not admissible for $s=2$  if $\eta_q$ is a generator of $I_{\m}$. 
\end{thm}

\begin{proof}
By the Ribet's exact sequence in Corollary \ref{cor:Ribetexactsequence}, we have
$$
\xymatrix{
0 \ar[r] & \Phi_{\ell} \ar[r] & (X/{\eta_q X})_{\ell} \ar[r] & \Psi^+_{\ell} \ar[r] & 0,
}
$$
where $\Phi:=\Phi_p(J_0(pr))$, $X:=X_p(J_0(pr))$, and $\Psi:=\Phi_q(J_0^{pq}(r))$.
By Propositions \ref{prop:heckecomp} and \ref{prop:ordercomp}, we have $\# \Phi_{\m}=\ell^n$, where $\ell^n$ is the power of $\ell$
exactly dividing $r+1$. By the result in \cite{Yoo14}, $X_{\m}$ is free of rank 1 over $\T_{\m}$ and
$(\T/I)_{\m} \simeq (\T_\m/{I_\m})\simeq \Z/{{\ell^n}\Z}$. 
Assume that $\eta_q$ is not a generator of $I_{\m}$. Then, we have
$$
\# (X/{\eta_q X})_{\m} = \# (\T_\m/{\eta_q \T_\m}) > \# (\T_\m/I_\m) = \ell^n. 
$$
Thus, we have $\Psi^+_{\fn} \neq 0$, where $\fn:=(\ell, U_p-1, U_q-1, U_r+1, ~\cI_0^{pq}(r)) \subseteq \T^{pq}(r)$; in other words, $\fn$ is maximal. If it is $r$-old, then the pair $(p, q)$ is admissible for $s=2$, which contradicts our assumption. Therefore, $\fn$ is $r$-new and by the Jacquet-Langlands correspondence, $\fn$ is new.

Assume further that $q \not\equiv 1 \modl$, $r \not\equiv -1 ~(\mathrm{mod}~ \ell^2)$, and $\eta_q$ is a generator of $I_{\m}$, i.e.,
$I_{\m}=(\eta_q)=\m_{\m}$.
Let $K$ be the kernel of the map $J_0^{pr}(1) \times J_0^{pr}(1) \rightarrow J_0^{pr}(q)$ by two degeneracy maps.
Then, as in Section \ref{sec:intersection}, we have
$$
0 \subseteq K^+_{\m} \subseteq (K^{\perp})^+_{\m} \subseteq V_{\m},
$$
where $V \simeq J_0^{pr}(1)[\eta_q]$.
Since by Proposition \ref{prop:skorobogatov}, $K^+$ contains the Skorobogatov subgroup of $J_0^{pr}(1)$ at $r$, which is of order $r+1$ \upto, we have
$\# K^+_{\m} \geq \ell$. Since $(\eta_q)=I_{\m}=\m_{\m}$, we have
$V_\m \simeq J_0^{pr}(1)[\eta_q]_{\m}=J_0^{pr}(1)[\m]$ and it is of dimension 2 over $\T^{pr}(1)/{\m}\simeq \F_{\ell}$
by Theorem \ref{thm:appmulti}. Because $V/{(K^{\perp})^+}$ is the $\Gm$-dual of $K^+$, we have $(K^{\perp}/K)^+_{\m}=0$. 
Since $K^{\perp}/K$ is isomorphic to the intersection of the $q$-old subvariety and the $q$-new subvariety of $J_0^{pr}(q)$ and $\m$ is not in the support of $K^{\perp}/K$, level raising does not occur by Theorem \ref{thm:shimintersection}. Thus, $\fn$ is not $q$-new. 
\end{proof}
\vv

\subsubsection{Examples} \label{sec:examples1}
When $N$ is a prime, $\eta_q$ is a local generator of an Eisenstein ideal $\cI_0(N) \subseteq \T(N)$ at $\m:=(\ell, ~\cI_0(N))$ 
if and only if $q \not\equiv 1 \pmod \ell$ and $q$ is not an $\ell^\th$  power modulo $N$; however, when $N$ is composite, we don't know such simple congruences to get local generation. 

Consider the easiest case. As in Theorem \ref{thm:s2t3suff}, we assume that $p \not \equiv 1 \pmod \ell$ and $r \equiv -1 \pmod \ell$. Assume further that $r \not \equiv -1 \pmod {\ell^2}$. In this case, we have $I_{\m}=\m \T_{\m}$.
 
Let $f(\tau) = \sum_{n\geq 1} a_n \cdot e^{2\pi i n\tau}$ be a newform of weight 2 for $\Gamma_0(pr)$ whose mod $\ell$ Galois representation is reducible such that $a_p=1$ and $a_r=-1$. 
If $a_q \equiv q+1 \pmod {\m^2}$, then $\eta_q:=T_q-q-1 \in \m^2$, so $\eta_q$ is not a generator of $I_{\m}$.
Moreover, in our examples below, all newforms are defined over $\Q$, i.e., $\T_{\m} = \Z_{\ell}$ and $\m=\ell \Z_{\ell}$.
Thus, $\eta_q$ is not a generator of $I_{\m}$ if and only if $a_q \equiv q+1 \pmod {\ell^2}$.
In the examples below, we follow the notation in Stein's table \cite{St}.

\begin{enumerate}
\item The admissibility for $s=2$ of $(2, q, 19)$ when $\ell=5$.

A newform $f$ of level $pr=38$ (as above) is $E[38, 2]$. Let $a_n$ be the eigenvalue of $T_n$ for $E[38, 2]$.
Then $a_q \equiv q+1 \pmod {25}$ when $q$ = 23, 41, 97, 101, 109, 113, 149, 151, 193, 199, 239, 241, 251, 257,
277, 347, 359, 431, and 479 for primes $q < 500$. Since only $(2, 151)$, $(2, 241)$, $(2, 251)$, and $(2, 431)$
are admissible for $s=2$, a triple $(2, q, 19)$ is admissible for $s=2$ if 
$$
q = 23, 41, 97, 101, 109, 113, 149, 193, 199, 239, 257, 277, 347, 359, \text{ or } 479
$$
for $q<500$.

\begin{rem}
A newform of level $2\times 23 \times 19$ with $a_2=a_{23}=1$ and $a_{19}=-1$ is 
$E[874, 8]$.
\end{rem}

\item The admissibility for $s=2$ of $(3, q, 19)$ when $\ell=5$.

A newform $f$ of level $pr=57$ (as above) is $E[57, 3]$. Let $b_n$ be the eigenvalue of $T_n$ for $E[57, 3]$.
Then $b_q \equiv q+1 \pmod {25}$ when $q$ = 41, 97, 101, 167, 197, 251, 257, 269, 313, 349, 409, 419, 431, and 491
for primes $q < 500$. Since only $(3, 41)$, $(3, 431)$, and $(3, 491)$ 
are admissible for $s=2$, a triple $(3, q, 19)$ is admissible for $s=2$ if
$$
q = 97, 101, 167, 197, 251, 257, 269, 313, 349, 409, \text{ or } 419
$$
for $q<500$.

\item The admissibility for $s=2$ of $(2, q, 29)$ when $\ell=5$.

A newform $f$ of level $pr=58$ (as above) is $E[58, 1]$. Let $c_n$ be the eigenvalue of $T_n$ for $E[58, 1]$.
Then $c_q \equiv q +1 \pmod {25}$ when $q$ = 89, 97, 137, 151, 181, 191, 223, 241, 251, 347, 367, 401, 431, 433, and 491
for primes $q < 500$. Since only $(2, 151)$, $(2, 241)$, $(2, 251)$, and $(2, 431)$  
are admissible for $s=2$, a triple $(2, q, 29)$ is admissible for $s=2$ if 
$$
q = 89, 97, 137, 181, 191, 223, 347, 367, 401, 433, \text{ or } 491
$$
for $q<500$.

\item The admissibility for $s=2$ of $(2, q, 13)$ when $\ell=7$.

A newform $f$ of level $pr=26$ (as above) is $E[26, 2]$. Let $d_n$ be the eigenvalue of $T_n$ for $E[26, 2]$.
Then $d_q \equiv q +1 \pmod {49}$ when $q$ = 43, 101, 223, 229, 233, 269, 307, 311, and 349 
for primes $q < 500$. Since a pair $(2, q)$ is not admissible  for $s=2$ when $q<500$, a triple $(2, q, 13)$ is admissible for $s=2$ if 
$$
q = 43, 101, 223, 229, 233, 269, 307, 311, \text{ or } 349 
$$
for $q<500$.

\end{enumerate}

\begin{rem}
In the last case, a pair $(2, q)$ is admissible for $s=2$ when $q = 631$ and $q=673$. As before, we have 
$$
d_{631} \equiv 1+631 \pmod {49} \text{ and } d_{691} \equiv 1+691 \pmod {49}. 
$$
In other words, computations tells that if a pair $(p, q)$ is admissible for $s=2$ then
$\eta_q$ is not a generator of $I_\m$. 
\end{rem}
\vv

\subsection{Admissible quadruples}
By Theorem \ref{thm:adms1}, a quadruple $(p, q, r, w)$ is admissible for $s=1$ if and only if 
$q\equiv r\equiv w\equiv -1 \modl$; and by Corollary \ref{cor:adms2}, a quadruple $(p, q, r, w)$ 
is admissible for $s=2$ if and only if $r \equiv w \equiv -1 \modl$;
moreover, by Theorem \ref{thm:admnec} and \ref{thm:admsuff}, a quadruple $(p, q, r, w)$ is admissible for $s=3$ if and only if $w \equiv -1 \modl$.

When $s=4$, we do not know what is the precise necessary and sufficient condition for admissibility. Instead, we provide a mild sufficient condition for more general cases:  
Let $s=t$ be an even integer greater than $2$. 
By Theorem \ref{thm:admnec},
if a $t$-tuple $(p_1, \dots, p_t)$ is admissible for $s=t$, $\ell \mid \prod_{i=1}^t (p_i-1)$. Thus, we may assume that $p:=p_1 \equiv 1 \modl$. Let 
\[
I:=(U_{p_i}-1, ~\cI_0^{D}(p) : 1 \leq i \leq t-1) \subseteq \T^D(p)^\new
\]
and $\m:=(\ell,~I)$, where $D=\prod_{i=2}^{t-1} p_i$. 
\begin{thm}\label{thm:s=t}
Assume that $t \geq 4$ is even.
Let $p:=p_1 \equiv 1 \modl$ and let $w:=p_t$. 
A $t$-tuple $(p_1, \dots, p_t)$ is admissible for $s=t$ if $\eta_w:=T_w-w-1$ is not a generator of $I_\m$.
\end{thm}

\begin{proof}
The idea is very similar to those used in the proof of Theorem \ref{thm:s2t2}. Let $X:=X_p(J_0^D(p))$ and $\Phi:=\Phi_p(J_0^D(p))$. Then, by the Ribet's exact sequence in Corollary \ref{cor:Ribetexactsequence}, 
\[
\xymatrix{
0 \ar[r] & \Phi_{\ell} \ar[r] & (X/{\eta_w X})_{\ell} \ar[r] & \Psi^+_{\ell} \ar[r] & 0,
}
\]
where $\Psi=\Phi_w(J_0^{Dpw}(1))$. Let 
\[
\fn:=(\ell,~U_{p_i}-1,~\cI_0^{Dpw}(1) : 1 \leq i \leq t) \subseteq \T^{Dpw}(1)
\]
be the ideal of $\T^{Dpw}(1)$ corresponding to $\m$. For admissibility, it suffices to show that $\Psi_\fn \neq 0$ by the Jacquet-Langlands correspondence. Accordingly, it suffices to show that 
\[
\# (X/{\eta_w X})_\ell \otimes \T_\m =\# (X/{\eta_w X}) \otimes \T_\m > \# (\Phi \otimes \T_\m)=\# (\Phi_\ell \otimes \T_\m),
\]
where $\T:=\T^D(p)^\new$. Note that $\Phi_\ell$ is cyclic and is annihilated by $I$ (Proposition \ref{prop:heckecomp}).
Thus, we have
\[
\#(\T/I)\otimes \T_\m \geq \# (\Phi_\ell \otimes \T_\m)=\# (\Phi_\ell).
\]
Since $X$ is a $\T$-module of rank 1 in the sense of Mazur (\cite[\ts II. 6]{M77}) (cf. \cite[p. 149]{R88}, \cite[Lemma 4.13]{Hm}), we have
\[
\# (X/{\eta_w X} \otimes \T_\m) \geq \#(\T_\m /{\eta_w \T_\m}).
\]
Moreover, if $\eta_w$ is not a generator of $I_\m$, then
\[
\#(\T_\m /{\eta_w \T_\m}) > \# (\T_\m/I_\m) = \# (\T/I) \otimes \T_\m.
\]
Combining all the inequalities above, the result follows.
\end{proof}

\subsubsection{Examples}\label{sec:examples2}
Let $p:=p_1$, $q:=p_2$, $r:=p_3$ and $w:=p_4$. We assume that $p \equiv 1 \pmod \ell$. Consider the easiest case: assume further that $(p-1)(q-1)(r-1)$ is not divisible by $\ell^2$. In this case, $I_{\m}=\m \T_{\m}$, where $\T:=\T^{qr}(p)$. (Here, we note that $\T_\m=\T^{qr}(p)_\m$ since the pair $(q, r)$ is not admissible for $s=2$.)
 
Let  $f(\tau) = \sum_{n\geq 1} a_n \cdot e^{2\pi i n\tau}$ be a newform of weight 2 for $\Gamma_0(pqr)$ whose mod $\ell$ Galois representation is reducible such that $a_p=a_q=a_r=1$. 
If $a_w \equiv w+1 \pmod {\m^2}$, then $\eta_w:=T_w-w-1 \in \m^2$.
Moreover, in our examples below, all newforms are defined over $\Q$, 
i.e., $\T_{\m} = \Z_{\ell}$ and $\m=\ell \Z_{\ell}$. 
Thus, $\eta_w$ is not a generator of $I_{\m}$ if and only if $a_w \equiv w+1 \pmod {\ell^2}$.
In the examples below, we follow the notation in Stein's table \cite{St}.

\begin{enumerate}
\item The admissibility for $s=4$ of $(11, 2, 3, w)$ when $\ell=5$.

A newform $f$ of level $pqr=66$ (as above) is $E[66,2]$. Let $a_n$ be the eigenvalue of $T_n$ for $E[66, 2]$.
Then $a_w \equiv w+ 1 \pmod {25}$ when $w$ = 47, 53, 97, 101, 103, 127, 151, 211, 271, 307, 317, and 431 for primes $w < 500$. Thus, a quadruple $(11, 2, 3, w)$ is admissible for $s=4$ if
$$
w = 47, 53, 97, 47, 53, 97, 101, 103, 127, 151, 211, 271, 307, 317, \text{ or } 431 
$$
for $w<500$.

\item The admissibility for $s=4$ of $(31, 2, 3, w)$ when $\ell=5$.

A newform $f$ of level $pqr=186$ (as above) is $E[186, 3]$. Let $b_n$ be the eigenvalue of $T_n$ for $E[186, 3]$.
Then $b_w \equiv w +1 \pmod {25}$ when $w$ = 19, 43, 59, 67, 71, 101, 109, 113, 131, 157, 181, 191, 227, 281, 283, 307, 331, 349, 359, 421, 431, and 443
for primes $w < 500$. Thus, a quadruple $(31, 2, 3, w)$ is admissible for $s=4$ if
$$
w = 19, 43, 59, 67, 71, 101, 109, 113, 131, 157, 181, 191, 
$$ 
$$
227, 281, 283, 307, 331, 349, 359, 421, 431, \text{ or }443
$$
for $w<500$.
\end{enumerate}
\vv

\subsection{Proof of Theorem \ref{thm:sufficientmain}} \label{sec:proof of 1.3}
The first and third statements are Theorem \ref{thm:admsuff}. The second is Theorem \ref{thm:s=t}. The fourth is Theorem \ref{thm:adms1}, and the fifth is Corollary \ref{cor:adms2}. The sixth is Theorem \ref{thm:ribetraising}, and the last is Theorem \ref{thm:yooraising}. \qed
\vv

\section{Conjectures on admissible tuples}\label{sec:conjecture}
In this section, we introduce a conjecture on the kernel of the degeneracy map between Jacobians of Shimura curves. This conjecture helps us to understand admissible tuples more. 

\subsection{The kernel of the degeneracy map}
In the article \cite{R84}, Ribet studied the kernel of the degeneracy map between Jacobians of modular curves. More specifically, let $\gamma_p$ be the map
$$
\gamma_p : J_0(N)\times J_0(N) \rightarrow J_0(Np)
$$
induced by two degeneracy maps. He proved that the kernel of $\gamma_p$ is an antidiagonal embedding of the Shimura subgroup of $J_0(N)$. Similarly, he proposed Conjecture \ref{conj:intro}.

The proof of the theorem on the kernel of $\gamma_p$ is based on the congruence subgroup property of $\SL_2(\Z[\frac{1}{r}])$. In the thesis \cite{Y13}, the author proved that 
Conjecture \ref{conj:intro} can be deduced from the congruence subgroup property of 
$$
\Gamma_r:=\left( \cO\otimes_{\Z} \Z \left[\frac{1}{r} \right] \right)^{\times,\,1},
$$
where $\cO$ is a maximal order of the quaternion algebra of discriminant $D$ over $\Q$ and $A^{\times,\,1}$ denotes the set of (reduced) norm 1 elements in $A$; so far, the congruence subgroup property of $\Gamma_r$ is an open problem. For more details, see \cite{PR10}.
\vv

\subsection{Proof of Theorems \ref{thm:1.5} and \ref{thm:1.6}}
If we assume that Conjecture \ref{conj:intro} holds, then we can prove another level raising methods, which is very useful to understand admissible tuples more.
\begin{proof}[Proof of Theorem \ref{thm:1.5}]
Note that $1\leq s \leq t$ because we assume that the $t$-tuple $(p_1, \dots, p_t)$ is admissible for $s$.
Let $D:=\prod_{i=1}^t p_i$ and 
$$
\m:=(\ell, U_{p_i}-1, U_{p_j}+1,~\cI_0^D(1)~:~1 \leq i \leq s \text{ and }~s+1 \leq j \leq t) \subseteq \T^D(1).
$$
By our assumption, $\m$ is maximal in $\T^D(1)$ and hence $J_0^D(1)[\m] \neq 0$.  By the same argument as in Section \ref{sec:intersection}, we have
$$
\Omega:=J_0^D(p)_{p\hyp\old} \cap J_0^D(p)_{p\hyp\new} = K^{\perp} / K,
$$
where $K$ is the kernel of $\gamma_p^{\mathrm{Sh}}$. 
By Conjecture \ref{conj:intro},
$K$ is an antidiagonal embedding of $\Sigma$. Therefore $(+1)$-eigenspace of $K$ is $\Sigma$ and $(-1)$-eigenspace of $K$ is trivial. This implies
\[
\Omega[U_p+1] \sim J_0^D(1)[T_p+p+1].
\]
Since $U_{p_i}$ acts on $\Sigma_{p_j}$ by $-1$ (if $i = j$) or $1$ (if $i \neq j$) (Proposition \ref{prop:heckeskorobogatov}), $\Sigma[\m]=0$ if $t-s \neq 1$. Thus,
\[
\Omega[\fn] \simeq J_0^D(1)[\m] \neq 0,
\]
where $\fn:=(\ell, U_{p_i}-1, U_{p_j}+1,~\cI_0^D(p)~:~0 \leq i \leq s \text{ and }~s+1 \leq j \leq t) \subseteq \T^D(p)$.

\begin{enumerate}
\item
Let $p:=p_{t+1}$ and assume that $p\equiv -1 \modl$.
Since $T_p-p-1 \in \m$ and $\ell \in \m$, we have $\eta:=T_p+p+1 \in \m$ as well. Therefore $\Omega[\fa] \simeq J_0^D(1)[\m] \neq 0$, where
$$
\fa:=(\ell,~U_{p_i}-1,~U_{p_j}+1,~\cI_0^D(p)~:~1 \leq i \leq s \text{ and }~s+1 \leq j \leq t+1) \subseteq \T^D(p).
$$
In other words, $\fa$ is $p$-new and hence $\fa$ is new in $\T({pD})$ by the Jacquet-Langlands correspondence. Therefore the $(t+1)$-tuple $(p_1, \dots, p_{t+1})$ is admissible for $s$.

Conversely, if a $(t+1)$-tuple $(p_1, \dots, p_{t+1})$ is admissible for $s$ (with $s\leq t$), then $p_{t+1} \equiv -1 \modl$ by Theorem \ref{thm:admnec}.

\item
Let $p:=p_0$ and assume that $s+2 \leq t$. As above since $t-s \geq 2$, we get $\Omega[\fn] \simeq J_0^D(1)[\m] \neq 0$, which implies that $\fn$ is $p$-new. Thus, the $(t+1)$-tuple $(p_0, \dots, p_t)$ is admissible for $(s+1)$.
\end{enumerate}
\end{proof}

\begin{rem}\label{rem:obstruction}
Here $K$ is the obstruction of level-raising. An important observation used above is the following:
if $K$ does not have support at $\m$, then we can automatically raise the level.
\end{rem}

\begin{proof}[Proof of Theorem \ref{thm:1.6}]
Note that the necessary conditions follow from Theorem \ref{thm:admnec}. We show that the conditions in the theorem are sufficient by induction. 
When $t \leq 4$, the result follows from Theorem \ref{thm:sufficientmain}.
Assume that the theorem holds for $(t-1)\leq 4$. 
\begin{enumerate}
\item
Assume that $s \neq t$ and $t$ is even. Assume further that $p_j \equiv -1 \modl$ for all $s+1 \leq j \leq t$. Then, if $s+1=t$ then the statement follows from Theorem \ref{thm:admsuff}. 
If $s+2 \leq t$, then the $(t-1)$-tuple $(p_2, \dots, p_t)$ is admissible for $(s-1)$ by induction hypothesis. By Theorem \ref{thm:yooraising}, the result follows.

\item 
Assume that $s+2 \leq t$ and $t$ is odd. Assume further that $p_j \equiv -1 \modl$ for all $s+1 \leq j \leq t$. Then, the $(t-1)$-tuple $(p_1, \dots, p_{t-1})$ is admissible for $s$ by induction hypothesis. Thus, the result follows from Theorem \ref{thm:1.5} because $s+2 \leq t$.

\end{enumerate}
Therefore the result holds by induction.
\end{proof}
\vv

\subsection{More on admissible triples for $s=2$} \label{sec: remaining cases}
The cases we have not considered are those with $t=s+1$ for $t\geq 5$ odd. 
Even though similar method to the case $(s, t)=(2, 3)$ works, we will not consider them. Instead, we obtain the following.

\begin{thm}\label{thm:s2t3 with conj}
Assume that Conjecture \ref{conj:intro} holds. If a pair $(p, q)$ is admissible for $s=2$ and $r \equiv -1 \modl$, then a triple $(p, q, r)$ is admissible for $s=2$.
\end{thm}
\begin{proof}
This is a special case of Theorem \ref{thm:1.5} (\ref{thm:1.5.1}). 
\end{proof}
\vv

\appendix 
\section{The component group of $J_0^D(Np)$ over $\F_p$} \label{appendix}
In their article \cite{DR73}, Deligne and Rapoport studied integral models of modular curves. Buzzard and Helm extended their result to the case of Shimura curves \cite[\ts 4]{Bu97} and \cite[Appendix]{Hm}. In this appendix, we discuss the special fiber $J_0^D(Np)_{/{\F_p}}$ of $J_0^D(Np)_{/{\Z}}$ over $\F_p$ for a prime $p$ not dividing $DN$ and the Hecke actions on its component group. Assume that $D$ is the product of an even number of distinct primes and $N$ is a square-free integer prime to $D$.

\subsection{The special fiber $J_0^D(Np)_{/{\F_p}}$}
\begin{prop}[Deligne-Rapoport model]
The special fiber $X_0^D(Np)_{/{\F_p}}$ consists of two copies of $X_0^D(N)_{/{\F_p}}$. They meet transversally at supersingular points.
\end{prop}

Let $S$ be the set of supersingular points of $X_0^D(Np)_{/{\F_p}}$. Then $S$ is isomorphic to the set of isomorphism classes of right ideals of an Eichler order of level $N$ of the definite quaternion algebra over $\Q$ of discriminant $Dp$ (cf. \cite[\textsection 3]{R90}). By the theory of Raynaud \cite{Ra70}, we have the following description of the special fiber $J_0^D(Np)_{/{\F_p}}$: It satisfies the following exact sequence
$$
\xymatrix{
0 \ar[r] & J^0 \ar[r] & J_0^D(Np)_{/{\F_p}} \ar[r] & \Phi_p(J_0^D(Np)) \ar[r] & 0,
}
$$
where $J^0$ is \textit{the identity component} and $\Phi_p(J_0^D(Np))$ is \textit{the component group}. Moreover,
$J^0$ is an extension of $J_0^D(N) \times J_0^D(N) _{/{\F_p}}$ by $T$, \textit{the torus of} $J_0^D(Np)_{/{\F_p}}$. The Cartier dual of $T$, $\Hom(T, \Gm)$, is \textit{the character group} $X:=X_p(J_0^D(Np))$. It is isomorphic to the group of degree 0 elements in the free abelian group $\Z^S$ generated by the elements of $S$. (Note that, the degree of an element in $\Z^S$ is the sum of its coefficients.) There is a natural pairing of $\Z^S$ such that
$$
\text{for any } s, t \in S, \quad \langle s, t \rangle := \frac{\#\Aut(s)}{2}\delta_{st},
$$
where $\delta_{st}$ is the Kronecker $\delta$-function. 
This pairing induces an injection $X \hookrightarrow \Hom(X, \Z)$ and the cokernel of it is isomorphic to $\Phi_p(J_0^D(Np))$ by Grothendieck \cite{Gro72}. We call the following exact sequence \textit{the monodromy exact sequence}:
\begin{equation} \label{eqn:monodromy}
\xymatrix{
0 \ar[r] & X \ar[r]^-{i} & \Hom(X, \,\Z) \ar[r] & \Phi_p(J_0^D(Np)) \ar[r] & 0.
}
\end{equation}
For more details, see \cite[\textsection 2 and 3]{R90} when $D=1$, and \cite[\ts 5]{Hm} for general cases.
\vv

\subsection{Hecke actions on $\Phi_p(J_0^D(Np))$}
By \cite[Proposition 3.8]{R90}, the Frobenius automorphism $\Frob_p$ on $X$ is equal to the operator $T_p$ on it. (His work can also be generalized to the case we consider without further difficulties once we know the results by Buzzard and Helm.) Note that $\Frob_p$ sends $s \in S$ to some other $s' \in S$, or might fix $s$. For elements $s, t \in S$ the map $i$ above sends $s-t$ to $\phi_s - \phi_t$, where
$$
\phi_s(x) := \langle s, x \rangle \quad \text{for any } x \in S.
$$
Thus in the group $\Phi_p(J_0^D(Np))$, we have $\phi_s = \phi_t$ for any $s, t \in S$. 
Note that we implicitly exclude the case where $D=N=1$. (In this case, see \cite[\textsection II. 11]{M77}.)
Since for all $s \in S$, the elements $ \frac{2}{\#\Aut(s)}\phi_s$ generate $\Hom(X, \Z)$ and $\#\Aut(s)$ is a divisor of $6$, we have  
$\Phi_p(J_0^D(Np)) \sim \Phi$ (see Notation \ref{notation}), where 
\[
\Phi:=\< \phi_s\> \subset \Phi_p(J_0^D(Np))
\]
is the cyclic subgroup generated by the image of $\phi_s$ for some $s \in S$ (cf. \cite[Proposition 3.2]{R90}).

\begin{prop}\label{prop:heckecomp}
For a prime divisor $r$ of $Dp$ (resp. $N$), $U_r-1$ (resp. $U_r-r$) annihilates $\Phi$. Moreover, for a prime $r$ not dividing $DNp$, $T_r-r-1$ annihilates $\Phi$.
\end{prop}
\begin{proof}
On $\Phi$, we have $\phi_s = \phi_t$. Thus, $U_p(\phi_s)=\phi_t=\phi_s$, where $t=\Frob_p(s)$. Since $S$ is isomorphic to the set of isomorphism classes of right ideals on an Eichler order of level $N$ in the definite quaternion algebra over $\Q$ of discriminant $Dp$, the set of supersingular points of $X_0^{Dp/q}(Nq)_{/{\F_q}}$ is again $S$ for a prime divisor $q$ of $D$; in other words, the character group of $J_0^{Dp/q}(Nq)_{/{\F_q}}$ does not depend on the choice of a prime divisor $q$ of $Dp$. (Hence the same is true for the component group by the monodromy exact sequence.) Using the same description as above, we have $U_q(\phi_s)=\phi_s$ for a prime divisor $q$ of $D$.

Since the degree of the map $U_r$ is $r$ for a prime divisor $r$ of $N$, $U_r(\phi_s)=\sum a_i \phi_{s_i}$ and $\sum a_i = r$. Therefore $U_r(\phi_s)=r\phi_s$ because in $\Phi$, $\phi_s = \phi_{s_i}$ for all $i$. Similarly, $T_r(\phi_s)=(r+1)\phi_s$ for a prime $r$ not dividing $DNp$.
\end{proof}
\vv

\subsection{The order of $\Phi$}
\begin{prop}\label{prop:ordercomp}
The order of $\Phi$ is equal to $\varphi(Dp)\psi(N)$ \upto.
\end{prop}
\begin{proof}
Let $n$ be the order of $\Phi$. Therefore, for any degree 0 divisor $t=\sum a_i s_i$, we have $n\phi_s(t)=0$. We decompose $n$ as a sum $\sum n_i$ for non-negative integers $n_i$. Then
$$
n\phi_s(t)= \left(\sum\limits_j n_j \phi_s \right) \left(\sum\limits_i a_i s_i \right) = \sum\limits_j n_j\left(\phi_{s_j}\left(\sum\limits_i a_i s_i\right)\right) = \sum\limits_j n_j a_j \frac{\#\Aut(s_j)}{2}=0.
$$
Therefore, for any $i \neq j$, we have $n_i \frac{\#\Aut(s_i)}{2}=n_j \frac{\#\Aut(s_j)}{2} = c$ by taking $t=s_i-s_j$. Since $n>0$, each $n_i$ is positive and it is equal to $\frac{2c}{\#\Aut(s_i)}$, where $2c$ is the smallest positive integer which makes $\frac{2c}{\#\Aut(s_i)}$ an integer for all $i$. Since $\#\Aut(s_i)$ divides $6$, $2c$ is a divisor of $6$. Thus, we have
$$
n = \sum\limits_{s \in S} \frac{6}{\#\Aut(s)}
$$
\upto. 

Recall Eichler's mass formula \cite[Corollary 5.2.3]{Vig80}.
\begin{prop}[Mass formula]
Let $S$ be the set of isomorphism classes of right ideals of an Eichler order of level $N$ in a definite quaternion algebra of discriminant $Dp$ over a number field $K$. Then
$$
\sum\limits_{s \in S} \frac{\# R^{\times}}{\#\Aut(s)} = 2^{1-d} \times | \zeta_K(-1) | \times h_K \times \varphi(Dp) \psi(N),
$$
where $R$ is the ring of integers of $K$, $\zeta_K$ is the Dedekind zeta function of $K$, $d$ is the degree of $K$ over $\Q$, and $h_K$ is the class number of $K$.
\end{prop}

In our case $K=\Q$, so $| \zeta_K(-1)|=\frac{1}{12}$, $d_K=h_K=1$, and $\#R^{\times}=2$. Thus, the result follows.
\end{proof}
\vv

\subsection{The degeneracy map between component groups} \label{app:deg}
Let $q$ be a prime not dividing $DNp$.
Let $\Phi$ (resp. $\Phi'$) be the cyclic subgroup of the component group of $J_0^D(Np)$ (resp. $J_0^D(Npq)$) at $p$ generated by the image of $\phi_s$ (resp. $\phi_{t}$) for some $s\in S$  (resp. $t \in S'$) as above, where $S$ (resp. $S'$) be the set of supersingular points of $X_0^D(Np)_{/{\F_p}}$ (resp. $X_0^D(Npq)_{/{\F_p}}$), i.e.,  $\Phi = \langle \phi_s \rangle$ (resp. $\Phi' = \langle \phi_t \rangle$).

Let $\gamma_q^{\mathrm{Sh}} : J_0^D(Np) \times J_0^D(Np) \rightarrow J_0^D(Npq)$ be the map defined by $\gamma_q^{\mathrm{Sh}} (a, b) = \alpha_q^*(a)+\beta_q^*(b)$, where
$\alpha_q, \beta_q$ are two degeneracy maps $X_0^D(Npq) \rightarrow X_0^D(Np)$. Then, $\gamma_q^{\mathrm{Sh}}$ induces a map $\gamma : \Phi \times \Phi \rightarrow \Phi'$. 

\begin{prop}\label{prop:degcomp}
Let $K'$ (resp. $C'$) be the kernel (resp. cokernel) of the map $\gamma$, i.e., we have an exact sequence
$$
\xymatrix{
0 \ar[r] & K' \ar[r] & \Phi \times \Phi \ar[r]^-{\gamma} & \Phi' \ar[r] & C' \ar[r] & 0.
}
$$
Then, $K' \sim \Phi$ and $C' \simeq \Z/{(q+1)\Z}$. 
\end{prop}
\begin{proof}
Since the degree of $\alpha_q^*$ is $q+1$, we have $\alpha_q^*(s) = \sum a_i t_i$ for some $t_i \in S'$, where $\sum a_i = q+1$. Because $\phi_{t_i}=\phi_t$ in $\Phi'$, we have $\alpha_q^*(\phi_s)=(q+1)\phi_t$.
By the same argument as above, we have $\beta_q^*(\phi_s)=(q+1)\phi_t$. Thus, the image of $\gamma$ is generated by $(q+1)\phi_t$ and $C' \simeq \Z/{(q+1)\Z}$.
Since $\alpha_q^*(\phi_s)=\beta_q^*(\phi_s)$, we have $(a,-a) \in K'$ for any $a \in \Phi$. By comparing their orders, we have $K' \sim \Phi$.
\end{proof}

\begin{cor}\label{cor:cokernelcomp}
For primes $r \mid Dp$ (resp. $r \mid N$, $r \nmid DNpq$), $U_r-1$ (resp. $U_r-r$, $T_r-r-1$) annihilates $C'$. Moreover, $U_q+1$ annihilates $C$.
\end{cor}
\begin{proof}
This follows from \ref{prop:heckecomp} because the order of $C$ is $q+1$.
\end{proof}

\begin{rem} \label{rem:C C'}
If $\ell\geq 5$ is prime, then the $\ell$-primary part of $K'$ (resp. $C'$) is equal to the one of the kernel $K$ 
(resp. cokernel $C$) of the map 
$$
\xymatrix{
\gamma_q^{\mathrm{Sh}} : \Phi_p(J_0^D(Np)) \times \Phi_p(J_0^D(Np)) \ar[r] &  \Phi_p(J_0^D(Npq))
}
$$
because $\Phi_p(J_0^D(Np))$ (resp. $\Phi_p(J_0^D(Npq))$) and $\Phi$ (resp. $\Phi'$) are equal up to 2-, 3- primary subgroups. 
\end{rem}
\vv

\section{Multiplicity one theorems}
In this section, we prove the multiplicity one theorems for Jacobians of modular and Shimura curves.
\subsection{The dimension of $J_0^{pr}(1)[\m]$} 
Let $J:=J_0^{pr}(1)$ be the Jacobian of the Shimura curve $X_0^{pr}(1)$ and $\T:=\T^{pr}(1)$ the Hecke ring in $\End(J)$. 
Assume that $p \not \equiv 1 \pmod \ell$ and $r \equiv -1 \pmod \ell$. Let
\begin{align*}
\fa:=(\ell,~U_p-1,~U_r+1,~\cI_0(pr)) &\subseteq \T(pr),\\
\fb:=(\ell,~U_r-1,~\cI_0(r)) &\subseteq \T(r), \\
\m:=(\ell, ~U_p-1, ~U_r+1, ~\cI_0) & \subseteq \T.
\end{align*}
Since $r\equiv -1 \modl$, $\fb=\T(r)$. Also, $\fa$ is neither $p$-old nor $r$-old. Therefore $\T_\m \simeq \T(pr)_{\fa}$. In this case we can prove the following multiplicity one theorem for $J[\m]$:

\begin{thm}[Ribet] \label{thm:appmulti}
$J[\m]$ is of dimension 2.
\end{thm}

For a proof, we need the following proposition.

\begin{prop} 
$\T_{\m}$ is Gorenstein.
\end{prop}

\begin{proof}
Let $Y$ be the character group of $J_{/\F_p}$. Then by Ribet \cite[Theorem 4.1]{R90}, there is 
an exact sequence
$$
\xymatrix{
0 \ar[r] & Y \ar[r] & L \ar[r] & X \oplus X \ar[r] & 0,
}
$$
where $L:=X_r(J_0(pr))$ (resp. $X:=X_r(J_0(r))$) is the character group of $J_0(pr)_{/{\F_r}}$ (resp. $J_0(r)_{/\F_r}$). 
Since $\fb=\T(r)$, $X_\fb=0$. Thus, we have $Y_{\m} \simeq L_{\fa}$. Since for $\fa$, multiplicity one theorem holds \cite[Theorem 4.5(2)]{Yoo14}, it implies that
$L_{\fa}$ is free of rank 1 over $\T({pr})_{\fa}$, i.e., $Y_{\m}$ is free of rank 1 over $\T_{\m}$. By Grothendieck \cite{Gro72}, 
there is a monodromy exact sequence
$$
\xymatrix{
0 \ar[r] & Y \ar[r] & \Hom(Y, \,\Z) \ar[r] & \Phi \ar[r] & 0,
}
$$
where $\Phi := \Phi_p(J)$ is the component group of $J_{/{{\F_p}}}$. After tensoring with $\Z_{\ell}$ over $\Z$, we have
$$
\xymatrix{
0 \ar[r] & Y\otimes \Z_{\ell} \ar[r] & \Hom(Y \otimes \Z_{\ell}, \,\Z_{\ell}) \ar[r] & \Phi_{\ell} \ar[r] & 0.
}
$$
Using an idempotent $e_{\m} \in \T_{\ell}:=\T \otimes \Z_{\ell}$, we have
$$
\xymatrix{
0 \ar[r] & Y_{\m} \ar[r] & \Hom(Y_{\m}, \,\Z_{\ell}) \ar[r] & \Phi_{\m} \ar[r] & 0.
}
$$
By the Ribet's exact sequence in Theorem \ref{thm:Ribetexactsequence}, we have
$$
\xymatrix{
0 \ar[r] & K \ar[r] & (X\oplus X)/(\d_p(X\oplus X)) \ar[r] & \Phi \ar[r] & C \ar[r] & 0,
}
$$ 
where $K \sim \Phi_r(J_0(r))$ and $C \sim \Z/{(p+1)\Z}$.
Since $r\equiv -1 \modl$, the first, second and fourth terms vanish after taking the completions at $\m$,.
Therefore we have $\Phi_{\m}=0$, which implies that
$Y_{\m} \simeq \Hom(Y_{\m}, \,\Z_{\ell})$ is a free self-dual $\T_\m$-module of rank $1$. Hence $\T_{\m}$ is Gorenstein.
\end{proof}

Now we prove the theorem above.
\begin{proof}[Proof of Theorem \ref{thm:appmulti}]
Let $J_{\m}:= \cup_n J[\m^n]$ be the $\m$-divisible group of $J$ and let $T_{\m}J$ be the Tate module of $J$ at $\m$, which is $\Hom(\Q_{\ell}/{\Z_{\ell}}, \,J_{\m})$ (cf. \cite[\textsection II. 7]{M77}). Then $T_{\m}J$ is free of rank 2 if and only if $J[\m]$ is of dimension 2 over $\T/{\m}$.
Since $J$ has purely toric reduction at $p$, there is an exact sequence
$$
\xymatrix{
0 \ar[r] & \Hom(Y/{\ell^n Y}, \,\mu_{\ell^n}) \ar[r] & J[\ell^n] \ar[r] & Y/{\ell^n Y} \ar[r] & 0
}
$$
for any $n \geq 1$ (cf. \cite[\textsection 3.3]{R76}).
By taking projective limit, we have 
$$
\xymatrix{
0 \ar[r] & \Hom(Y\otimes \Z_{\ell}, \,\Z_{\ell}(1)) \ar[r] & T_{\ell}J \ar[r] & Y\otimes \Z_{\ell} \ar[r] & 0,
}
$$
where $\Z_{\ell}(1)$ is the Tate twist. By applying the idempotent $e_{\m}$ (of the decomposition $\T_{\ell}=\prod_{\m \mid \ell} \T_{\m}$), we get
$$
\xymatrix{
0 \ar[r] & \Hom(Y_{\m}, \,\Z_{\ell}(1)) \ar[r] & T_{\m}J \ar[r] & Y_{\m} \ar[r] & 0.
}
$$
Since $Y_{\m}$ is free of rank 1 over $\T_{\m}$, $T_{\m}J$ is free of rank 2 over $\T_{\m}$, and hence $J[\m]$ is of dimension 2.
\end{proof}
\vv

\subsection{The dimension of $J_0(pqr)[\m]$}
Let $J:=J_0(pqr)$ and $\T:=\T({pqr})$.
Let $L:=X_p(J)$ be the character group of $J$ at $p$ and  
$\m:=(\ell, ~U_p-1,~U_q-1,~U_r+1,~\cI_0) \subseteq \T$. 
Assume that $\ell$ does not divide $(p-1)(q-1)$ and $r \equiv -1 \modl$. Then, we have the following.

\begin{thm}\label{thm:multipqr}
$L/{\m L}$ is of dimension 1 over $\T/{\m}$ and $J[\m]$ is of dimension 2.
\end{thm}
\begin{proof}
By \cite[Theorem 4.2(2)]{Yoo14}, we have $\dim J[\m]=2$.
 
Let $T$ be the torus of $J$ at $p$. Note that $J[\m]$ is a non-trivial extension of $\muell$
by $\zell$, which is ramified only at $r$.
Since $\Frob_p$ acts by $pU_p$ on $T$, $T[\m]$ cannot contain $\zell$.
Since the dimension of $J[\m]$ is $2$, $T[\m]$ is at most of dimension 1. On the other hand,  
$\T^{p\hyp\new}$ acts faithfully on $T$ and $\m$ is $p$-new because a pair $(p, r)$ is admissible for $s=1$. Accordingly, the dimension of $T[\m]$ is at least 1. Therefore $L/{\m L}$, which is the dual space of $T[\m]$, is of dimension 1.
\end{proof}
\vv

\section{The Skorobogatov subgroup of $J_0^{pD}(N)$ at $p$}\label{sec:Skorobogatov}
In his article \cite{Sk05}, Skorobogatov introduced ``Shimura coverings'' of Shimura curves. 
Let $B$ be a quaternion algebra over $\Q$ of discriminant $pD$ such that $B\otimes_{\Q} \R \simeq M_2(\R)$. 
Let $\cO$ be an Eichler order of $B$ of level $N$, and set $\Gamma_0^{pD}(N):=\cO^{\times, 1}$, the set of reduced norm 1 elements in $\cO$. Let $I_p$ be the unique two-sided ideal of $\cO$ of reduced norm $p$. Then, $1+I_p \subseteq \Gamma_0^{pD}(N)$ 
and it defines a covering $X \rightarrow X_0^{pD}(N)$. By Jordan \cite{Sk05}, there is an unramified subcovering $X \rightarrow X_p \rightarrow X_0^{pD}(N)$ whose Galois group is $\Z/{((p+1)/{\epsilon(p))}}$, where $\epsilon(p)$ is a divisor of $6$. (About $\epsilon(p)$ when $N=1$, see \cite[p. 781]{Sk05}.) Since unramifed abelian coverings of $X_0^{pD}(N)$ correspond to subgroups of $J_0^{pD}(N)$, we define the Skorobogatov subgroup of $J_0^{pD}(N)$ from $X_p$.

\begin{defn}
The \textit{Skorobogatov subgroup} $\Sigma_p$ of $J_0^{pD}(N)$ at $p$ is the subgroup of $J_0^{pD}(N)$ which 
corresponds to the unramified covering $X_p$ of $X_0^{pD}(N)$ above.
\end{defn}

These subgroups have similar properties to Shimura subgroups. 

\begin{prop}\label{prop:heckeskorobogatov}
On $\Sigma_p$, $U_p$ (resp. $U_q$, $T_k$) acts by $-1$ (resp. $1$, $k+1$) for primes $q \mid D$ and $k \nmid pDN$.
\end{prop}
\begin{proof}
By using moduli theoretic description of $X_0^{pD}(N)$, 
the complex points of $X$ classifies $(A, P)$ where $A$ is a false elliptic curve with level $N$ structure and $P$ is a generator of $A[I_p]$.
Since the level structures at primes $r$ dividing $DN$ are compatible with the level structure at $p$, which gives rise to our covering $X$, the Atkin-Lehner involution $w_r$ acts trivially on the covering group. This gives the action of $U_q$ when $q$ divides $D$ because $U_q=w_q$. 

For primes $k$ not dividing $pDN$, let $\alpha_k$ and $\beta_k$ denote two degeneracy maps from $X_0^{pD}(Nk)$ to $X_0^{pD}(N)$.
Then, we have $T_k={(\beta_k)}_*\alpha_k^*={(\beta_k)}_*w_k\beta_k^*={(\beta_k)}_*\beta_k^*=k+1$
because the image of the Skorobogatov subgroup at $p$ by the degeneracy maps lies in the Skorobogatov subgroup at $p$ and $w_k$ acts trivially on it. 

Consider $U_p$ on $\Sigma_p$. The map $U_p$ sends $(A, P)$ to $(A/{A[I_p], Q)}$, where $\langle P, Q \rangle=\zeta_p$ for some fixed primitive $p^\th$ root of unity $\zeta_p$ and the pairing $\langle -,~-\rangle$ on $A[I_p] \times A[p]/{A[I_p]}$. (About this pairing, see \cite{Bu97}.) For $\sigma$ in the covering group of $X \rightarrow X_0^{pD}(N)$,
it sends $(A, P)$ to $(A, \sigma P)$. Thus $U_p \sigma U_p^{-1}=\sigma^{-1}$, which implies $U_p$ acts by $-1$ on $\Sigma_p$.
\end{proof}

\begin{rem} 
It might be easier than above if you consider the actions of $w_p$ on the group of $2\times 2$ matrices as in Calegari and Venkatesh. 
See \cite[p. 29]{CV12}.
\end{rem}

\begin{prop}\label{prop:skorobogatov}
Let $K$ be the kernel of the map
$$
J_0^{pr}(1) \times J_0^{pr}(1) \rightarrow J_0^{pr}(q)
$$
induced by two degeneracy maps $\alpha_q^*$ and $\beta_q^*$. Then $K$ contains an antidiagonal embedding of $\Sigma_r$.
\end{prop}

\begin{proof}
Let $\Sigma_r$ (resp. $\Sigma$) be the Skorobogatov subgroup of $J_0^{pr}(1)$ (resp. $J_0^{pr}(q)$) at $r$.
Since $w_q$ acts trivially on $\Sigma$ and the image of $\Sigma_r$ by $\alpha_q^*$ lies in $\Sigma$, 
$\alpha_q^*(a)+\beta_q^*(-a)= \alpha_q^*(a)+w_q(\alpha_q^*(-a))=\alpha_q^*(a)-\alpha_q^*(a)=0$.
Thus, $K$ contains $\{(a, -a) \in J_0^{pr}(1)^2~:~ a \in \Sigma_r \}$.
\end{proof}

\begin{rem}
Since $K$ contains an antidiagonal embedding of $\Sigma_r$, if $r\equiv -1 \modl$ we have $K[\m] \neq 0$, where
$\m : = (\ell, ~U_p-1,~U_q-1,~U_r+1, ~\cI_0^{pr}(q)) \subset \T^{pr}(q)$.
\end{rem}

\bibliographystyle{annotation}

\begin{thebibliography}{99}

\bibitem{ARS12} A. Agashe, K. Ribet, and W. Stein, \emph{The modular degree, congruence primes, and multiplicity one}, Number Theory, Analysis and Geometry, In Memory of Serge Lang, Springer (2012), 19--50.

\bibitem{BD1} N. Billerey and L.V. Dieulefait, \emph{Explicit large image theorems for modular forms}, Jour. Lond. Math. Soc. (2), 89(2) (2014), 499--523.

\bibitem{BM1} N. Billerey and R. Menares, \emph{On the modularity of reducible mod $\ell$ Galois representations}, Math. Res. Lett. {23(1)} (2016), 15--41.

\bibitem{BM2} N. Billerey and R. Menares, \emph{Strong modularity of reducible Galois representations}, to appear in Trans. Amer. Math. Soc., available at \url{https://arxiv.org/abs/1604.01173} (2016).

\bibitem{Bu97} K. Buzzard, \emph{Integral models of certain Shimura curves}, Duke Math. Journal, Vol {87}, no 3. (1997), 591--612.

\bibitem{CV12} F. Calegari and A. Venkatesh, \emph{A torsion Jacquet-Langlands correspondence}, preprint available at \url{https://arxiv.org/abs/1604.01173} (2012).

\bibitem{Ce76} I.V. Cerednik, \emph{Uniformization of algebraic curves by discrete arithmetic subgroups of $\mathrm{PGL}_2(k_w)$ with compact quotients (in Russian)}, Math. Sb. {100} (1976), 59--88. Translation in Math. USSR Sb. {29} (1976), 55--78.

\bibitem{DR73} P. Deligne and M. Rapoport, \emph{Les sch\'emas de modules de courbes elliptiques}, Modular functions of one variable II, Lecture notes in Math., Vol. {349} (1973), 143--316.


\bibitem{DT94}  F. Diamond and R. Taylor, \emph{Non-optimal levels of mod $\ell$ modular representations}, Invent. Math., Vol {115} (1994), 435--462.


\bibitem{Dr76} V. Drinfeld, \emph{Coverings of $p$-adic symmetric regions (in Russian)}, Funkts. Anal. Prilozn {10} (1976), 29--40. Translation in Funct. Anal. Appli. {10} (1976), 107--115.


\bibitem{Gro72} A. Grothendieck, \emph{SGA 7 I. Expose IX}, Lecture Notes in Math., Vol {288} (1972), 313--523.

\bibitem{Hm} D. Helm, \emph{On maps between modular Jacobians and Jacobians of Shimura curves}, Israel Journal of Math. Vol. {160} (2007), 61--117.

\bibitem{Ig59} J.-I. Igusa, \emph{Kroneckerian model of fields of elliptic modular functions}, American Journal of Math., Vol {81} (1959), 561--577.

\bibitem{KM85} N. Katz and B. Mazur, \emph{Arithmetic moduli of elliptic curves}, Princeton Univ. Press, Princeton, Annals of Math. Studies {108} (1985).

\bibitem{M77} B. Mazur, \emph{Modular curves and the Eisenstein Ideal}, Publications Math. de l'I.H.\'E.S., tome {47} (1977), 33--186.



\bibitem{PR10} G. Prasad and A. Rapinchuk, \emph{Developments on the congruence subgroup problem after the work of Bass, Milnor and Serre}, John Milnor's collected works {Vol. V}, AMS (2010), 307--325.


\bibitem{Ra70} M. Raynaud, \emph{Sp\'ecialization du foncteur de Picard}, Publications Math. de l'I.H.\'E.S., tome {38} (1970), 27--76.

\bibitem{R76} K. Ribet, \emph{Galois action on division points of Abelian varieties with real multiplications}, American Journal of Math., Vol. {98} (1976), 751--804.

\bibitem{R84} K. Ribet, \emph{Congruence relations between modular forms}, Proceeding of the International Congress of Math., Vol. {1}, {2} (Warsaw, 1983) (1983), 503--514.






\bibitem{R90} K. Ribet, \emph{On modular representations of $\Gal(\overline \Q/{\Q})$ arising from modular forms}, Invent. Math. {100}, no. 2 (1990), 431--476.

\bibitem{R88} K. Ribet, \emph{Torsion points on $J_0(N)$ and Galois representations}, Arithmetic theory of elliptic curves (Cetraro, 1997), Lectures in Math. 1716, Springer, Berlin (1999), 145--166.




\bibitem{R10} K. Ribet, \emph{Non-optimal levels of mod $\ell$ reducible Galois representations}, CRM Lecture note available at \url{http://math.berkeley.edu/~ribet/crm.pdf} (2010).




\bibitem{Se87} J.-P. Serre, \emph{Sur les repr\'esentations modulaires de degr\'e 2 de $\Gal(\overline{\Q}/{\Q})$}, Duke Math. Journal, Vol {54}, no 1. (1987), 179--230. 

\bibitem{Sk05} A. Skorobogatov, \emph{Shimura coverings of Shimura curves and the Manin obstruction}, Mathematical Research Letter, Vol {12} (2005), 779--788. 


\bibitem{St} W. Stein, \emph{The modular forms database: Tables}, available at \url{http://wstein.org/Tables/an.html}.





\bibitem{Vig80} M.-F. Vign\'eras, \emph{Arithm\'{e}tique des alg\`{e}bres de quaternions}, Lecture notes in Math., Vol. {800} (1980).

\bibitem{Y13} H. Yoo, \emph{Modularity of residually reducible Galois representations and Eisenstein ideals}, UC Berkeley Thesis (2013).

\bibitem{Yoo14} H. Yoo, \emph{The index of an Eisenstein ideal and multiplicity one}, Mathematische Zeitschrift Vol. {282(3)} (2016), 1097--1116.

\bibitem{Yoo15a} H. Yoo, \emph{On Eisenstein ideals and the cuspidal group of $J_0(N)$}, Israel Journal of Math. Vol. {214} (2016), 359--377.

\end{thebibliography}

\end{document}